\documentclass[aos,preprint]{imsart}
\setattribute{journal}{name}{}

\RequirePackage[OT1]{fontenc}
\RequirePackage{amsthm,amsmath}
\RequirePackage[numbers]{natbib}
\RequirePackage[colorlinks,citecolor=blue,urlcolor=blue]{hyperref}
\usepackage{amssymb,calc}
\usepackage{thmtools}
\usepackage{mathrsfs}
\usepackage{mathtools}
\usepackage{bm}
\usepackage{bbm}
\usepackage{enumerate}
\usepackage{rotating}
\RequirePackage{cleveref}
\usepackage{hyphenat}

\usepackage{graphicx,epstopdf,epsfig}
\usepackage{caption,subcaption}
\usepackage{pgfplots}

\startlocaldefs

\numberwithin{equation}{section}

\theoremstyle{plain}
\newtheorem{thm}{Theorem}[section]
\newtheorem{prop}[thm]{Proposition}
\newtheorem{lem}[thm]{Lemma}

\theoremstyle{definition}
\newtheorem*{defn}{Definition}
\newtheorem{exa}{Example}
\newtheorem*{exa*}{Example}

\theoremstyle{remark}
\newtheorem*{rem}{Remark}

\crefname{thm}{Theorem}{Theorems}
\crefname{prop}{Proposition}{Propositions}
\crefname{lem}{Lemma}{Lemmas}
\crefname{coro}{Corollary}{Corollaries}
\crefname{add}{Addendum}{Addendums}
\crefname{alg}{Algorithm}{Algorithms}
\crefname{proc}{Procedure}{Procedures}
\crefname{exe}{Exercise}{Exercises}
\crefname{exa}{Example}{Examples}
\crefname{prob}{Problem}{Problems}
\crefname{section}{Section}{Sections}
\crefname{subsection}{Section}{Sections}
\crefname{appendix}{Appendix}{Appendices}

\DeclareMathOperator{\Var}{Var}
\DeclareMathOperator{\Cov}{Cov}
\DeclareMathOperator{\as}{as}
\DeclareMathOperator{\oas}{as\ast}
\DeclareMathOperator*{\conv}{\mathchoice{%
	\,\longrightarrow\,}{
	\rightarrow}{
	\rightarrow}{
	\rightarrow}
}

\endlocaldefs

\allowdisplaybreaks[3]

\begin{document}

\begin{frontmatter}
\title{Asymptotic Theory of $L$\hyp{}Statistics and Integrable Empirical Processes\thanksref{T1}}
\runtitle{$L$\hyp{}Statistics and Integrable Processes}
\thankstext{T1}{The previous version was circulated with the title ``Switching to the New Norm: From Heuristics to Formal Tests using Integrable Empirical Processes.''}

\begin{aug}
\author{\fnms{Tetsuya} \snm{Kaji}\ead[label=e1]{tkaji@chicagobooth.edu}}

\runauthor{T. Kaji}

\affiliation{The University of Chicago}

\address{The University of Chicago\\Booth School of Business\\
5807 South Woodlawn Avenue\\Chicago, IL 60637\\
\printead{e1}\\
\phantom{E-mail:\ }}
\end{aug}

\begin{abstract}
This paper develops asymptotic theory of integrals of empirical quantile functions with respect to random weight functions, which is an extension of classical $L$\hyp{}statistics.
They appear when sample trimming or Winsorization is applied to asymptotically linear estimators.
The key idea is to consider empirical processes in the spaces appropriate for integration.
First, we characterize weak convergence of empirical distribution functions and random weight functions in the space of {\em bounded integrable} functions.
Second, we establish the delta method for empirical quantile functions as {\em integrable} functions.
Third, we derive the delta method for $L$\hyp{}statistics.
Finally, we prove weak convergence of their bootstrap processes, showing validity of nonparametric bootstrap.
\end{abstract}



\end{frontmatter}


\section{Introduction}

We derive the asymptotic distribution of the statistics of the form
\[
	\int_0^1m(\mathbb{Q}_n)d\mathbb{K}_n,
\]
where $m:\mathbb{R}\to\mathbb{R}$ is a known continuously differentiable function, $\mathbb{Q}_n:(0,1)\to\mathbb{R}$ an empirical quantile function of a random variable $X_i$, and $\mathbb{K}_n:(0,1)\to\mathbb{R}$ a random Lipschitz function that depends on $\{X_i\}$.
This is a generalization of the classical $L$\hyp{}statistics \cite{ms1992,s1997,s2017,sw1986} to allow for integration with respect to random processes $\mathbb{K}_n$.%
\footnote{\cite{s1997} allows integration on a random interval but not with respect to a random process.}

This type of statistics appears, for example, when sample trimming or Winsorization is applied to asymptotically linear estimators.
Let us collectively call sample trimming and Winsorization {\em sample adjustments}.
If sample adjustments are made conditional on the values of $X_i$, $\mathbb{K}_n$ is a nonrandom function and it falls within the framework of classical $L$\hyp{}statistics.
If sample adjustments are made on variables other than $X_i$, $\mathbb{K}_n$ becomes random and it affects the asymptotic distribution of the $L$\hyp{}statistics.
In economics, this occurs as the parameters of interest (what $L$\hyp{}statistics estimate) often differ from the variables whose outliers we are concerned.
In such cases, dependence of $\mathbb{K}_n$ can be difficult to handle directly.

This paper gives both high\hyp{}level and low\hyp{}level conditions for weak convergence of the $L$\hyp{}statistics, derives the asymptotic distribution formula, and verifies validity of nonparametric bootstrap.
The innovation of this paper lies in considering empirical processes in the space of integrable functions.
The literature on empirical processes has largely focused on uniform convergence irrespective of the intended statistical application.
As $L$\hyp{}statistics are integrals of empirical processes, we (partly) renounce uniform convergence and instead require integrability, which buys us substantial benefits in dealing with $L$\hyp{}statistics.
\footnote{In applying the empirical process theory to $L$\hyp{}statistics, Van der Vaart \cite[Chapter 22]{v1998} states that ``[this approach] is preferable in that it applies to more general statistics, but it{\ldots}does not cover the simplest $L$\hyp{}statistic: the sample mean.'' Our empirical process theory overcomes this problem.}


Our theoretical development is summarized as follows.
By integration by parts, we expect
\begin{multline*}
	\sqrt{n}\biggl[\int m(\mathbb{Q}_n)d\mathbb{K}_n-\int m(Q)dK\biggr]\\
	\begin{aligned}
	&=\int\sqrt{n}[m(\mathbb{Q}_n)-m(Q)]d\mathbb{K}_n+\int \!m(Q)d[\sqrt{n}(\mathbb{K}_n-K)]\\
	&\approx\int\sqrt{n}[m(\mathbb{Q}_n)-m(Q)]dK-\int\sqrt{n}(\mathbb{K}_n-K)dm(Q).
	\end{aligned}
\end{multline*}
First, we consider $\sqrt{n}(\mathbb{F}_n-F)$ and $\sqrt{n}(\mathbb{K}_n-K)$ as elements in the space of bounded integrable functions with respect to appropriate measures and derive conditions for weak convergence therein (\cref{AppendixA}).
Second, we establish the functional delta method for the ``inverse map,'' $F\mapsto m(F^{-1})=m(Q)$, from the space of bounded integrable functions to the space of integrable functions, which shows weak convergence of $\sqrt{n}[m(\mathbb{Q}_n)-m(Q)]$ as an integrable process (\cref{AppendixB}).%
\footnote{This paper is presumably the first to show weak convergence of (possibly unbounded) empirical quantile processes in $L_1$ on the untruncated domain $(0,1)$.}
Third, we develop the functional delta method for the map, $(Q,K)\mapsto\int m(Q)dK$, from the spaces of integrable and bounded integrable functions to a Euclidean space, establishing weak convergence of $L$\hyp{}statistics (\cref{AppendixC}).
Finally, we develop conditions for nonparametric bootstrap for the processes and $L$\hyp{}statistics (\cref{AppendixD}).

The theory of this paper was originally motivated by the following problem of formalizing outlier robustness analyses in economics.

\begin{exa}[Outlier Robustness Analysis] \label{exa:1}
Applied researchers often want to examine whether a small portion of outliers affect the regression outcomes \cite{ajkkm2016,ajr2001,anrr2016,abb2010,frw2007}.
The common heuristic practice in economics is to compare two estimators $\hat{\beta}_1$ and $\hat{\beta}_2$, where $\hat{\beta}_1$ is estimated with the full sample and $\hat{\beta}_2$ with the sample that excludes outliers, against the standard error of $\hat{\beta}_1$.
However, since $\hat{\beta}_1$ and $\hat{\beta}_2$ share largely overlapping samples, their difference tends to be small simply because of their strong positive correlation.
To account for this, it is more appropriate to compare the difference $\hat{\beta}_1-\hat{\beta}_2$ to its own variance, as opposed to the marginal variance of $\hat{\beta}_1$.
This calls for the joint distribution of $\hat{\beta}_1$ and $\hat{\beta}_2$.

Consider linear regression $y_i=x_i\beta+\varepsilon_i$ with $\mathbb{E}[x_i\varepsilon_i]=0$.
The ordinary least squares (OLS) estimator of $\beta$ is $\hat{\beta}_1=\bigl(\frac{1}{n}\sum_{i=1}^n x_i^2\bigr)^{-1}\frac{1}{n}\sum_{i=1}^n x_iy_i$, so its asymptotic distribution depends on that of the average of $x_iy_i$.
However, $x_iy_i$ is usually not the quantity whose outliers are of natural concern, but rather, $x_i$ \cite{ajr2001}, $y_i$ \cite{ajkkm2016,bdh2014}, or $\hat{\varepsilon}_i$ \cite{anrr2016} is.
Then, conditional on the value of $x_iy_i$, the probability that the observation is deemed as an outlier is probabilistic.

Suppose we remove the 2\% tail observations of $x_i$ and $y_i$.
Let $w_i$ be $1$ if $x_{(\lceil0.02n\rceil)}\leq x_i\leq x_{(\lceil0.98n\rceil)}$ and $y_{(\lceil0.02n\rceil)}\leq y_i\leq y_{(\lceil0.98n\rceil)}$, and $0$ otherwise.%
\footnote{Winsorization can also be accommodated by appropriately defining $w_i$.}
The outlier\hyp{}removed estimator $\hat{\beta}_2$ is $\bigl(\frac{1}{n}\sum_{i=1}^n x_i^2w_i\bigr)^{-1}\frac{1}{n}\sum_{i=1}^n x_iy_iw_i$.
Through the quantile transform, we can write
\[
	\hat{\beta}_1=\biggl(\frac{1}{n}\sum_{i=1}^nx_i^2\biggr)^{-1}\int_0^1\mathbb{Q}_n(u)du, \quad
	\hat{\beta}_2=\biggl(\frac{1}{n}\sum_{i=1}^nx_i^2w_i\biggr)^{-1}\int_0^1\mathbb{Q}_n(u)d\mathbb{K}_n(u),
\]
where $\mathbb{Q}_n$ is the empirical quantile function of $x_iy_i$ and $\mathbb{K}_n$ a random weight function whose derivative is $w_i$ for $u\in(\mathbb{F}_n(x_iy_i)-1/n,\mathbb{F}_n(x_iy_i)]$ for the empirical distribution function $\mathbb{F}_n$ of $x_iy_i$.
Then $\mathbb{K}_n$ is random for each fixed value of $x_iy_i$, which affects the asymptotic distribution of the integrals.

In \cref{suppB}, we revisit the outlier robustness analysis in \cite{anrr2016}.
\end{exa}

The rest of the paper is organized as follows.
\cref{sec:setup} defines the setup.
\cref{AppendixA} develops the theory of weak convergence of bounded integrable processes.
\cref{AppendixB} establishes Hadamard differentiability of the inverse map.
\cref{AppendixC} shows Hadamard differentiability of the $L$\hyp{}statistics.
\cref{AppendixD} verifies validity of nonparametric bootstrap.
\cref{sec:proofs} contains proofs.
\cref{app} contains supporting lemmas and an empirical application.

\section{The Setup} \label{sec:setup}

Let $X_i$ be i.i.d.\ scalar random variables and $w_i$ be possibly random weights whose distribution is bounded but can depend on all of $\{X_i\}$. Consider a statistic of the form
\(
	\hat{\beta}\vcentcolon=\frac{1}{n}\sum_{i=1}^n m(X_i)w_i=\frac{1}{n}\sum_{i=1}^n m(X_{(i)})w_{(i)},
\)
where $m$ is a continuously differentiable function, $X_{(i)}$ is an order statistic such that $X_{(1)}\leq X_{(2)}\leq\cdots\leq X_{(n)}$, and $w_{(i)}$ is ordered according to the order of $X_i$.
Let $\mathbb{Q}_n(u)\vcentcolon=X_{(i)}$ and $d\mathbb{K}_n(u)\vcentcolon=w_{(i)}$, $u\in(\frac{i-1}{n},\frac{i}{n}]$, be the empirical quantile function of $X_i$ and the random weight function.
With these,
\[
	\hat{\beta}=\int_0^1 m(\mathbb{Q}_n(u))d\mathbb{K}_n(u).
\]
Denote by $\mathbb{F}_n(x)\vcentcolon=\frac{1}{n}\sum_{i=1}^n\mathbbm{1}\{X_i\leq x\}$ the empirical distribution function of $X_i$ and define the inverse of a nondecreasing function $f:\mathbb{R}\to\mathbb{R}$ by $f^{-1}(y)\vcentcolon=\inf\{x\in\mathbb{R}:f(x)\geq y\}$.
Then, $\mathbb{Q}_n$ equals $\mathbb{F}_n^{-1}$.

The aim of this paper is to derive the joint distribution of finitely many such quantities $(\hat{\beta}_1,\dots,\hat{\beta}_d)$ for possibly different $m$, $\{X_i\}$, and $\{w_i\}$.
For this, we proceed in four steps:
\begin{enumerate}[i.]
	\item Give conditions for convergence of $\sqrt{n}(\mathbb{F}_n-F)$ and $\sqrt{n}(\mathbb{K}_n-K)$ to Gaussian processes as bounded integrable processes.
	\item Show convergence of $\sqrt{n}(\mathbb{Q}_n-Q)$ to a Gaussian process as an integrable process via a functional delta method from $\mathbb{F}_n$ to $\mathbb{Q}_n$.
	\item Show convergence of $L$\hyp{}statistics via a functional delta method from $(\mathbb{Q}_n,\mathbb{K}_n)$ to $\int m(\mathbb{Q}_n)d\mathbb{K}_n$.
	\item Show bootstrap convergence for $\sqrt{n}(\mathbb{F}_n-F)$ and $\sqrt{n}(\mathbb{K}_n-K)$.
\end{enumerate}

\section{Convergence of Bounded Integrable Processes} \label{AppendixA}

Define the space of bounded integrable functions as follows.

\begin{defn}
Let $(T, \mathcal{T}, \mu)$ be a measure space where $T$ is an arbitrary set, $\mathcal{T}$ a $\sigma$\hyp{}field on $T$, and $\mu$ a $\sigma$\hyp{}finite signed measure on $\mathcal{T}$.
Let $\mathbb{L}_\mu$ be the space of bounded and $\mu$\hyp{}integrable functions $z : T \to \mathbb{R}$ with the norm
\[
	\|z\|_{\mathbb{L}_\mu} \vcentcolon= \|z\|_T \vee \|z\|_\mu \vcentcolon= \biggl( \sup_{t\in T} |z(t)| \biggr) \vee \biggl( \int_T |z| |d\mu| \biggr),
\]
where $|d\mu|$ represents integration with respect to the total variation measure.%
\end{defn}

For sums of i.i.d.\ random variables such as $\sqrt{n}(\mathbb{F}_n-F)$, it is straightforward to prove weak convergence in $\mathbb{L}_\mu$ by the combination of classical central limit theorems (CLTs) \cite{vw1996}.

\begin{prop} \label{prop:A:F}
Let $(\mathbb{R},\mathfrak{B}(\mathbb{R}),\mu)$ be a $\sigma$-finite Borel measure on $\mathbb{R}$.
For a probability distribution $F$ on $\mathbb{R}$ such that $\int_{\mathbb{R}}\sqrt{F(1-F)}|d\mu|<\infty$, the empirical process $\sqrt{n}(\mathbb{F}_n-F)$ converges weakly in $\mathbb{L}_\mu$ to a Gaussian process with mean zero and covariance function $\Cov(x,y)=F(x\wedge y)-F(x)F(y)$.
%
\end{prop}

\begin{rem}
For an increasing function $m$, $\int_{-\infty}^\infty\sqrt{F(1-F)}dm<\infty$ is equivalent to $\|m(X)\|_{2,1}\vcentcolon=\int_0^\infty\sqrt{\Pr(|m(X)|>t)}dt<\infty$ \cite{bgm1999}.
Moreover, if $m(X)$ has a $(2+c)$th moment for some $c>0$, we have $\|m(X)\|_{2,1}<\infty$.
\end{rem}

For processes not given as sums of i.i.d.\ variables such as $\sqrt{n}(\mathbb{K}_n-K)$, we need direct conditions for weak convergence.
As in classical literature, we characterize weak convergence in $\mathbb{L}_\mu$ by asymptotic tightness plus marginal convergence.
Following \cite{vw1996}, we consider a {\em net} $X_\alpha$ indexed by an arbitrary directed set, rather than a sequence $X_n$ indexed by natural numbers.
We also allow the sample space to be different for each element in a net, $X_\alpha:\Omega_\alpha\to\mathbb{L}_\mu$.
Finally, we allow each element in the net to be not necessarily measurable.
When we write $X(t)$ for a map $X : \Omega \to \mathbb{L}_\mu$, $t$ is understood to be an element of $T$ and we regard $X(t)$ as a map from $\Omega$ to $\mathbb{R}$ indexed by $T$; when we explicitly use $\omega \in \Omega$ in the discussion, we write $X(t,\omega)$.

\begin{thm} \label{FfdAo8VYYjsJz}
Let $X_\alpha : \Omega_\alpha \to \mathbb{L}_\mu$ be arbitrary. Then, $X_\alpha$ converges weakly to a tight limit if and only if $X_\alpha$ is asymptotically tight and marginals $(X_\alpha(t_1), \dots, X_\alpha(t_k))$ converge weakly for every finite subset $t_1, \dots, t_k$ of $T$.
If $X_\alpha$ is asymptotically tight and its marginals converge weakly to the marginals $(X(t_1), \dots, X(t_k))$ of a stochastic process $X$, then there is a version of $X$ with sample paths in $\mathbb{L}_\mu$ and $X_\alpha \leadsto X$.
\end{thm}

Weak convergence of marginals can be established by classical results such as CLTs in Euclidean spaces.
The question is asymptotic tightness. We characterize this with uniform equicontinuity and equiintegrability.


\begin{defn}
For a $\mu$\hyp{}measurable semimetric $\rho$ on $T$,%
\footnote{We call a semimetric {\em $\mu$\hyp{}measurable} if every open set induced is measurable with respect to $\mu$.}
the net $X_\alpha : \Omega_\alpha \to \mathbb{L}_\mu$ is {\em asymptotically uniformly $\rho$\hyp{}equicontinuous and $(\rho,\mu)$\hyp{}equiintegrable in probability} if for every $\varepsilon,\eta>0$ there exists $\delta>0$ such that
\begin{multline*}
	\limsup_\alpha P^\ast \biggl( \sup_{t\in T} \biggl[ \biggl( \sup_{\rho(s,t)<\delta} |X_\alpha(s)-X_\alpha(t)| \biggr) \\
	\vee \biggl( \int_{0<\rho(s,t)<\delta} |X_\alpha(s)| |d\mu(s)| \biggr) \biggr] > \varepsilon \biggr) < \eta.
\end{multline*}
\end{defn}


The following result characterizes asymptotic tightness in $\mathbb{L}_\mu$.

\begin{thm} \label{dJMhMPoGhEIJvZ}
The following are equivalent.
\begin{enumerate}[i.]
	\item \label{thmA6:1} A net $X_\alpha : \Omega_\alpha \to \mathbb{L}_\mu$ is asymptotically tight.
	\item \label{thmA6:2} $X_\alpha(t)$ is asymptotically tight in $\mathbb{R}$ for every $t \in T$, $\|X_\alpha\|_\mu$ is asymptotically tight in $\mathbb{R}$, and for every $\varepsilon,\eta>0$ there exists a finite $\mu$\hyp{}measurable partition $T=\bigcup_{i=1}^k T_i$ such that
	\begin{multline} \label{partition}
		\limsup_\alpha P^\ast \Biggl( \biggl[ \sup_{1\leq i\leq k} \sup_{s,t\in T_i} |X_\alpha(s)-X_\alpha(t)| \biggr] \\
		\vee \sum_{i=1}^k \inf_{x\in\mathbb{R}} \int_{T_i} |X_\alpha-x| |d\mu| > \varepsilon \Biggr) < \eta.
	\end{multline}
	\item \label{thmA6:3} $X_\alpha(t)$ is asymptotically tight in $\mathbb{R}$ for every $t \in T$ and there exists a $\mu$\hyp{}measurable semimetric $\rho$ on $T$ such that $(T,\rho)$ is totally bounded and $X_\alpha$ is asymptotically uniformly $\rho$\hyp{}equicontinuous and $(\rho,\mu)$\hyp{}equiintegrable in probability.
\end{enumerate}
\end{thm}

\begin{rem}
The condition ``$0 < \rho(s,t)$'' allows for the point masses in $\mu$ and plateaus in $X_\alpha$.
In (\ref{partition}), this corresponds to ``$-x$.''
\end{rem}

Now we turn to conditions for $\sqrt{n}(\mathbb{K}_n-K)$.
The following is a special case of $\mathbb{L}_\mu$ suitable for $\mathbb{K}_n$.

\begin{defn}
Let $Q:(0,1)\to\mathbb{R}$ be an integrable increasing function and let $\mathbb{L}_Q$ be the space of functions $\kappa:(0,1)\to\mathbb{R}$ with the norm
\[
	\|\kappa\|_{\mathbb{L}_Q}\vcentcolon=\|\kappa\|_{Q,\infty}\vee\|\kappa\|_Q\vcentcolon=\biggl(\sup_{u\in(0,1)}|(|Q|\vee1)(u)\kappa(u)|\biggr)\vee\biggl(\int_0^1|\kappa|dQ\biggr).
\]
Let $\mathbb{L}_{Q,M}\subset\mathbb{L}_Q$ be the subset of Lipschitz functions with Lipschitz constants bounded by $M$.
\end{defn}

The following lemma gives a low\hyp{}level condition for $\sqrt{n}(\mathbb{K}_n-K)$ to converge in $\mathbb{L}_Q$.
Roughly, if $|Q|^{2+c}$ is integrable, then $\frac{X_\alpha}{u^r(1-u)^r}\leadsto\frac{X}{u^r(1-u)^r}$ in the uniform norm for some $r>\frac{1}{2+c}$ implies $X_\alpha\leadsto X$ in $\mathbb{L}_Q$.

\begin{lem} \label{X4cQDt7zeDM7E9hTYSZ}
Let $Q : (0,1) \to \mathbb{R}$ be an increasing function in $L_{2+c}$ for some $c>0$.
If for a net of processes $X_\alpha : \Omega_\alpha \to \mathbb{L}_Q$ there exists $r>\frac{1}{2+c}$ such that for every $\eta>0$ there exists $M$ satisfying
\[
	\limsup_\alpha P^\ast \biggl( \biggl\| \frac{X_\alpha}{u^r(1-u)^r} \biggr\|_\infty > M \biggr) < \eta,
\]
then there exists a semimetric $\rho$ on $(0,1)$ such that $(0,1)$ is totally bounded, $X_\alpha$ is asymptotically uniformly $\rho$\hyp{}equicontinuous in probability, and $X_\alpha$ is asymptotically $(\rho,Q)$\hyp{}equiintegrable in probability.
\end{lem}

This implies that sample adjustments based on fixed quantiles satisfy the condition.
For example, let $X_1,\dots,X_n$ be i.i.d.\ continuous random variables and $X_{1,n},\dots,X_{m,n}$ be their subset selected by some (possibly random) criterion.
Then, if the empirical process of the subset converges weakly uniformly to a smooth distribution, then $\sqrt{n}(\mathbb{K}_n-K)$ converges weakly in $\mathbb{L}_Q$.

\begin{prop} \label{p34AI8lskauT}
Let $U_1, \dots, U_n$ be independent uniformly distributed random variables on $(0,1)$ and $w_{1,n}, \dots, w_{n,n}$ random variables bounded by $M$ whose distribution can depend on $U_1, \dots, U_n$ and $n$.
Define
\(
	\mathbb{F}_n(u) \vcentcolon= \frac{1}{n} \sum_{i=1}^n \mathbbm{1}\{U_i \leq u\}$ and $
	\mathbb{G}_n(u) \vcentcolon= \frac{1}{n} \sum_{i=1}^n w_{i,n} \mathbbm{1}\{U_i \leq u\}.
\)
Let $I(u)\vcentcolon=u$ and assume that $K(u)\vcentcolon=\lim_{n\to\infty} \mathbb{E}[\mathbb{G}_n(u)]$ exists and is Lipschitz differentiable.
If $\sqrt{n}(\mathbb{F}_n-I)$ and $\sqrt{n}(\mathbb{G}_n-K)$ converges weakly jointly in $L_\infty$, then for
\[
	\mathbb{K}_n(u) \vcentcolon= \frac{1}{n} \sum_{i=1}^n w_{i,n} \mathbbm{1} \bigl\{ 0 \vee \bigl( nu - n\mathbb{F}_n(U_i) + 1 \bigr) \wedge 1 \bigr\},
\]
we have $\sqrt{n}(\mathbb{K}_n-K)$ converge weakly in $\mathbb{L}_Q$ for every increasing function $Q\in L_{2+c}$ for every $c>0$.
\end{prop}



\section{Convergence of Quantile Processes as Integrable Processes} \label{AppendixB}

For a smooth function $m$ for which $m(X)$ has sufficient moments, we establish weak convergence of $\sqrt{n}(m(\mathbb{Q}_n)-m(Q))$ to a Gaussian process.
If $m$ is identity, the (unweighted) empirical quantile process converges weakly in $L_1$ on the entire domain $(0,1)$, without truncating the tails, even if $Q$ is an unbounded function.
Interestingly, this point has been overlooked in the literature, which mostly concerned uniform convergence of either bounded or weighted quantile processes \cite{k2002,ch1993,chs1993,cchm1986,sw1986}.

In particular, we show differentiability of the inverse map as a functional from $\mathbb{L}_\mu$ to $L_1$.
Note that $\mathbb{E}[m(X)]=\int mdF=-\int Fdm$ in terms of $F$ and $\mathbb{E}[m(X)]=\int m(Q)du$ in terms of $Q$.
Therefore, the appropriate space for $F$ is the following special case of $\mathbb{L}_\mu$ while the space for $Q$ is a standard $L_1$.

\begin{defn}
Let $m:\mathbb{R}\to\mathbb{R}$ be a nondecreasing continuously differentiable function.
Let $\mathbb{L}_m$ be the space of Borel-measurable functions $z:\mathbb{R}\to\mathbb{R}$ with limits $z(\pm\infty) \vcentcolon= \lim_{x\to\pm\infty} z(x)$ and the norm
\[
	\|z\|_{\mathbb{L}_m} \vcentcolon= \|z\|_\infty \vee \|z\|_m \vcentcolon= \biggl( \sup_{x\in\mathbb{R}} |z(x)| \biggr) \vee \biggl( \int_{-\infty}^\infty |\tilde{z}| dm \biggr)
\]
where $\tilde{z}(x)\vcentcolon=z(x)-z(-\infty)\mathbbm{1}\{x<0\}-z(+\infty)\mathbbm{1}\{x\geq0\}$.
Denote by $\mathbb{L}_{m,\phi}$ the subset of $\mathbb{L}_m$ of monotone cadlag functions with $z(-\infty)=0$ and $z(+\infty)=1$.
\end{defn}


\begin{defn}
Let $\mathbb{B}$ be the space of ladcag functions $z:(0,1)\to\mathbb{R}$ with the norm
\(
	\|z\|_{\mathbb{B}} \vcentcolon= \int_0^1 |z(u)| du.
\)
\end{defn}



\begin{thm}[Inverse map] \label{XCNkYqAFjk2IK6F}
Let $m : \mathbb{R} \to \mathbb{R}$ be a continuously differentiable function and $F \in \mathbb{L}_{m,\phi}$ a distribution function on (an interval of) $\mathbb{R}$ that has at most finitely many jumps and is otherwise continuously differentiable with strictly positive density $f$.
Then, the map $\phi \circ \psi : \mathbb{L}_{m,\phi} \to \mathbb{B}$, $\phi \circ \psi(F) \vcentcolon= m(Q)$, is Hadamard differentiable at $F$ tangentially to the set $\mathbb{L}_{m,0}$ of all continuous functions in $\mathbb{L}_m$. The derivative is given by
\(
	(\phi\circ\psi)_F'(z)\vcentcolon=-(m' z/f)\circ Q.
\)
\end{thm}

The main conclusion of this section is summarized as follows.

\begin{prop} \label{dqt032FkNegj}
Let $m : \mathbb{R} \to \mathbb{R}$ be a continuously differentiable function.
For a distribution function $F$ on (an interval of) $\mathbb{R}$ that has at most finitely many jumps and is otherwise continuously differentiable with strictly positive density $f$ such that $\int_{\mathbb{R}}\sqrt{F(1-F)}|dm|<\infty$, the process $\sqrt{n}(m(\mathbb{Q}_n)-m(Q))$ converges weakly in $\mathbb{B}$ to a Gaussian process with mean zero and covariance $\Cov(s,t) = m'(Q(s))Q'(s)m'(Q(t))Q'(t) (s\wedge t-st)$.
\end{prop}


\section{Convergence of $L$\hyp{}statistics} \label{AppendixC}

We seek conditions under which the integral of a stochastic process with respect to another stochastic process converges weakly.
This is an extension of Wilcoxon statistics \cite[Section 3.9.4.1]{vw1996} that allows unbounded integrands.


\begin{thm}[Wilcoxon statistic] \label{thm:wilcoxon}
For each fixed $M$, the maps $\lambda:\mathbb{B}\times\mathbb{L}_{Q,M}\to\mathbb{R}$ and $\tilde{\lambda}:\mathbb{B}\times\mathbb{L}_{Q,M}\to L_\infty(0,1)^2$,
\(
	\lambda(Q,K)\vcentcolon=\int_0^1QdK $ and $
	\tilde{\lambda}(Q,K)(s,t)\vcentcolon=\int_s^tQdK,
\)
are Hadamard differentiable at every $(Q,K)\in\mathbb{B}\times\mathbb{L}_{Q,M}$ uniformly over $\mathbb{L}_{Q,M}$.
The derivative maps are
\(
	\lambda_{Q,K}'(z,\kappa)\vcentcolon=\int_0^1Qd\kappa+\int_0^1zdK $ and $
	\tilde{\lambda}_{Q,K}'(z,\kappa)(s,t)\vcentcolon=\int_s^tQd\kappa+\int_s^tzdK,
\)
where $\int Qd\kappa$ is defined via integration by parts if $\kappa$ is of unbounded variation.
\end{thm}

Now we are ready to give the main conclusion of this paper.

\begin{prop}[$L$\hyp{}statistic] \label{EkaRgpsiHieiNgk4NFP}
Let $m_1, m_2 : \mathbb{R} \to \mathbb{R}$ be continuously differentiable functions and $F : \mathbb{R}^2 \to [0,1]$ be a distribution function on (a rectangular of) $\mathbb{R}^2$ with marginal distributions $(F_1, F_2)$ that have at most finitely many jumps and are otherwise continuously differentiable with strictly positive marginal densities $(f_1, f_2)$ such that $m_1(X_1)$ and $m_2(X_2)$, $(X_1,X_2) \sim F$, have $(2+c)$th moments for some $c>0$.
Along with i.i.d.\ random variables $X_{1,1}, \dots, X_{n,1}$ and $X_{1,2}, \dots, X_{n,2}$, let $w_{1,n,1}, \dots, w_{n,n,1}$ and $w_{1,n,2}, \dots, w_{n,n,2}$ be random variables bounded by $M$ whose distribution can depend on $n$, $X_{1,1}, \dots, X_{n,1}$, and $X_{1,2}, \dots, X_{n,2}$ such that the empirical distributions of $X_{i,1}$, $X_{i,2}$, $w_{i,n,1} X_{i,1}$, and $w_{i,n,2} X_{i,2}$ converge uniformly jointly to continuously differentiable functions.
Then,
\begin{multline*}
	\sqrt{n} \begin{pmatrix} \mathbb{E}_n[m_1(X_{i,1}) w_{i,n,1}] - \mathbb{E}[m_1(X_{i,1}) w_{i,n,1}] \\ \mathbb{E}_n[m_2(X_{i,2}) w_{i,n,2}] - \mathbb{E}[m_2(X_{i,2}) w_{i,n,2}] \end{pmatrix} \\
	= \sqrt{n} \begin{pmatrix} \int_0^1 m_1(\mathbb{Q}_{n,1}) d\mathbb{K}_{n,1} - \int_0^1 m_1(Q_1) dK_1 \\ \int_0^1 m_2(\mathbb{Q}_{n,2}) d\mathbb{K}_{n,2} - \int_0^1 m_2(Q_2) dK_2 \end{pmatrix}
\end{multline*}
where
\(
	\mathbb{K}_{n,j}(u) \vcentcolon= \frac{1}{n} \sum_{i=1}^n w_{i,n,j} \mathbbm{1} \bigl\{ 0 \vee \bigl( nu - n\mathbb{F}_{n,j}(X_i) + 1 \bigr) \wedge 1 \bigr\}
\)
and
\(
	K_j(u) \vcentcolon= \lim_{n\to\infty} \mathbb{E}[w_{i,n,j} \mid F_j(X_{i,j}) \leq u],
\)
converge weakly in $\mathbb{R}^2$ to a normal vector $(\xi_1,\xi_2)$ with mean zero and (co)variance
\begin{multline*}
	\Cov(\xi_j, \xi_k) = \int_0^1 \int_0^1 m_j'(Q_j(s))Q_j'(s)m_k'(Q_k(t))Q_k'(t) \times \\
		\Bigl( [F_{jk}^Q(s,t)-st] + [K_{jk}(s,t) F_{jk}^Q(s,t)-stK_j(s)K_k(t)] \\
		{}- K_j(s) [F_{jk}^Q(s,t)-st] - K_k(t) [F_{jk}^Q(s,t)-st] \Bigr) ds dt,
\end{multline*}
where $F_{jk}^Q(s,t)\vcentcolon=\Pr(X_{i,j}\leq Q_j(s), X_{i,k}\leq Q_k(t))$ and $K_{jk}(s,t)\vcentcolon=\lim_{n\to\infty}$ $\mathbb{E}[w_{i,n,j}w_{i,n,k} \mid X_{i,j} \leq Q_j(s), X_{i,k} \leq Q_k(t)]$.
If $F$ has no jumps, this equals
\begin{multline*}
	\Cov(\xi_j, \xi_k) = \int_{-\infty}^\infty \int_{-\infty}^\infty \Bigl( [1-K_j^F(x)-K_k^F(y)][F_{jk}(x,y)-F_j(x)F_k(y)] \\
		+ [K_{jk}^F(x,y) F_{jk}(x,y)-K_j^F(x)K_k^F(y)F_j(x)F_k(y)] \Bigr) dm_j(x) dm_k(y),
\end{multline*}
where $F_{jk}(x,y)\vcentcolon=\Pr(X_{i,j}\leq x, X_{i,k}\leq y)$ and $K_{jk}^F(x,y)\vcentcolon=\lim_{n\to\infty}\mathbb{E}[w_{i,n,j}$ $w_{i,n,k} \mid X_{i,j} \leq x, X_{i,k} \leq y]$.
If $m_j$ and $m_k$ are known, this can be consistently estimated by its sample analogue
\begin{gather*}
	\widehat{\Cov(\xi_j, \xi_k)} = \int_{-\infty}^\infty \int_{-\infty}^\infty \Bigl( [1-\mathbb{K}_{n,j}^F(x)-\mathbb{K}_{n,k}^F(y)][\mathbb{F}_{n,jk}(x,y)-\mathbb{F}_{n,j}(x)\mathbb{F}_{n,k}(y)] \\
		{}+ [\mathbb{K}_{n,jk}^F(x,y) \mathbb{F}_{n,jk}(x,y)-\mathbb{K}_{n,j}^F(x)\mathbb{K}_{n,k}^F(y)\mathbb{F}_{n,j}(x)\mathbb{F}_{n,k}(y)] \Bigr) dm_j(x) dm_k(y),
\end{gather*}
where $\mathbb{F}_{n,jk}(x,y)\vcentcolon=\mathbb{E}_n[\mathbbm{1}\{X_{i,j}\leq x, X_{i,k}\leq y\}]$ and $\mathbb{K}_{n,jk}^F(x,y)\vcentcolon=\mathbb{E}_n[w_{i,n,j}$ $w_{i,n,k} \mid X_{i,j} \leq x, X_{i,k} \leq y]$.
\end{prop}


\section{Convergence of Bootstrap Processes} \label{AppendixD}

We establish validity of nonparametric bootstrap, viz., conditional weak convergence of the bootstrap processes. The {\em bootstrap process} for $\mathbb{F}_n$ is given by
\begin{align*}
	\hat{\mathbb{Z}}_n(x)\vcentcolon=\sqrt{n}(\hat{\mathbb{F}}_n-\mathbb{F}_n)(x)
	&\vcentcolon=\frac{1}{\sqrt{n}}\sum_{i=1}^n(M_{ni}-1)\mathbbm{1}\{X_i\leq x\}\\
	&\hphantom{:}=\frac{1}{\sqrt{n}}\sum_{i=1}^n(M_{ni}-1)(\mathbbm{1}\{X_i\leq x\}-F(x))
\end{align*}
where $M_{ni}$ is the number of times $X_i$ is drawn in the bootstrap sample.
We show that $\hat{\mathbb{Z}}_n$ converges weakly to the same limit as $\mathbb{Z}_n\vcentcolon=\sqrt{n}(\mathbb{F}_n-F)$ conditional on $\{X_i\}$.
As in \cite[Chapter 3.6]{vw1996}, we proceed as follows: since $M_{ni}$ sums up to $n$, it is slightly dependent on each other; we replace $M_{ni}$ with {\em independent} Poisson random variables $\xi_i$ by showing equivalence of weak convergence of $\hat{\mathbb{Z}}_n$ and of the {\em multiplier process} $\mathbb{Z}_n'\vcentcolon=n^{-1/2}\sum\xi_i(\mathbbm{1}\{X_i\leq x\}-F)$ (\cref{lem:A:Poissonization}); then, we prove unconditional convergence of $\mathbb{Z}_n'$ (randomness comes from both $X_i$ and $\xi_i$) by symmetrization (\cref{lem:multiplier:uncond}); finally, we show convergence of $\mathbb{Z}_n'$ conditional on $\mathbb{Z}_n$ (randomness only comes from $\xi_i$) by discretizing $\mathbb{Z}_n'$ (\cref{lem:multiplier:cond}).
We observe that many proofs in \cite[Chapters 2.9, 3.6, and A.1]{vw1996} carry over to $\mathbb{L}_\mu$, so we will not reproduce the entire argument but prove steps that require modification.

In addition, we establish conditional weak convergence of the bootstrap process for $\mathbb{K}_n$. We restrict attention to sample adjustments by quantiles and write its bootstrap process in terms of empirical processes (\cref{lem:A:bootK}).

The following shows conditional convergence of $\mathbb{Z}_n'$ as in \cite[Theorem 2.9.6]{vw1996}.
Other lemmas are given in \cref{sec:proofs:bootstrap}.

\begin{lem} \label{lem:multiplier:cond}
Let $\xi_1,\dots,\xi_n$ be i.i.d.\ random variables with mean $0$, variance $1$, and $\|\xi\|_{2,1}<\infty$, independent of $X_1,\dots,X_n$.
For a probability distribution $F$ on $\mathbb{R}$ such that $\int_{\mathbb{R}}\sqrt{F(1-F)}|d\mu|<\infty$, the process $\mathbb{Z}_n'(x)=n^{-1/2}\sum_{i=1}^n\xi_i[\mathbbm{1}\{X_i\leq x\}-F(x)]$ satisfies
\(
	\sup_{h\in\text{\rm BL}_1(\mathbb{L}_\mu)}\bigl|\mathbb{E}_\xi h(\mathbb{Z}_n')-\mathbb{E}h(\mathbb{Z})\bigr|\conv0
\)
in outer probability, and the sequence $\mathbb{Z}_n'$ is asymptotically measurable.
\end{lem}

These results show that nonparametric bootstrap works for $\sqrt{n}(\mathbb{F}_n-F)$ and $\sqrt{n}(\mathbb{Q}_n-Q)$.
We also show validity for $\sqrt{n}(\mathbb{K}_n-K)$ by representing $\mathbb{K}_n$ as a function of ``$\mathbb{F}_n$" and ``$\mathbb{G}_n$'' in \cref{p34AI8lskauT}.

\begin{lem} \label{lem:A:bootK}
Let $U_1,\dots,U_n$ be independent uniformly distributed random variables on $(0,1)$ and $\xi_1,\dots,\xi_n$ be i.i.d.\ random variables with mean $0$, variance $1$, and $\|\xi\|_{2,1}<\infty$, independent of $U_1,\dots,U_n$.
Define the bootstrap empirical process of $U$ by
\(
	\mathbb{F}_n'(u)\vcentcolon=\frac{1}{n}\sum_{i=1}^n\xi_i\mathbbm{1}\{U_i\leq u\},
\)
and let $w_{i,n}'$ be the indicator of whether $U_i$ is above the $\alpha$\hyp{}quantile of the bootstrap sample, that is, $w_{i,n}'\vcentcolon=\mathbbm{1}\{U_i>\mathbb{F}_n'^{-1}(\alpha)\}$.
Define
\(
	\mathbb{G}_n'(u)\vcentcolon=\frac{1}{n}\sum_{i=1}^n\xi_iw_{i,n}'\mathbbm{1}\{U_i\leq u\}.
\)
Then, for $F(u)=0\vee u\wedge1$ and $G(u)=0\vee(u-\alpha)\wedge(1-\alpha)$,
\begin{gather*}
	\sup_{h\in\textrm{BL}_1(L_\infty)}\bigl|\mathbb{E}_\xi h\bigl(\sqrt{n}(\mathbb{F}_n'-F)\bigr)-\mathbb{E}h\bigl(\sqrt{n}(\mathbb{F}_n-F)\bigr)\bigr|\conv0, \\
	\sup_{h\in\textrm{BL}_1(L_\infty)}\bigl|\mathbb{E}_\xi h\bigl(\sqrt{n}(\mathbb{G}_n'-G)\bigr)-\mathbb{E}h\bigl(\sqrt{n}(\mathbb{G}_n-G)\bigr)\bigr|\conv0
\end{gather*}
in outer probability, and $\sqrt{n}(\mathbb{F}_n'-F)$ and $\sqrt{n}(\mathbb{G}_n'-G)$ are asymptotically measurable.
\end{lem}


Altogether, nonparametric bootstrap works for $L$\hyp{}statistics when sample adjustment is based on empirical quantiles.

\begin{prop}[Validity of nonparametric bootstrap] \label{prop:A:bootstrap}
In addition to assumptions in \cref{{EkaRgpsiHieiNgk4NFP}}, assume that $w_{i,n,j}$ represents sample adjustments based on a finite number of fixed quantiles.%
\footnote{The assumption on convergence must be extended to jointly over all processes.}
Then, the joint distribution of $(\hat{\beta}_1, \dots, \hat{\beta}_d)$ can be consistently estimated by nonparametric bootstrap. 
\end{prop}


\section{Proofs} \label{sec:proofs}


\subsection{Convergence of Bounded Integrable Processes}

\begin{proof}[Proof of \cref{prop:A:F}]
Marginal convergence is trivial.
By \cite[Example 2.5.4]{vw1996}, $\sqrt{n}(\mathbb{F}_n-F)$ converges weakly in $L_\infty$.
In light of \cite[Proposition 2.1.11]{vw1996} and \cref{lem:joint}, it suffices to show that for $Z(x)\vcentcolon=\mathbbm{1}\{X\leq x\}-F(x)$, (i)
\(
	\Pr(\|Z\|_\mu > t) = o(t^{-2})
\)
and (ii)
\(
	\int_{\mathbb{R}} \sqrt{\mathbb{E}[Z^2]} |d\mu| < \infty.
\)
(ii) follows since $\int_{\mathbb{R}}\sqrt{\mathbb{E}[Z^2]}|d\mu|=\int_{\mathbb{R}}\sqrt{F(1-F)}|d\mu|$.
Let $m(x)\vcentcolon=\int_{(-\infty,x]}|d\mu|$ and $\tilde{F}(x)\vcentcolon=F(x)-\mathbbm{1}\{0\leq x\}$.
Note that (ii) implies that $m(X)$ has variance.
Writing $Z(x)=\mathbbm{1}\{X \leq x\}-\mathbbm{1}\{0\leq x\}-\tilde{F}(x)$, we find
\(
	\|Z\|_\mu \leq |m(X)-m(0)| + \int_{\mathbb{R}} |\tilde{F}| |d\mu|.
\)
The second term is a finite constant if (ii) holds.
Thus, (i) holds if $\tilde{F} \circ m^{-1}(t) = o(t^{-2})$, which is the case if $m(X)$ has variance. 
Thus, $\sqrt{n}(\mathbb{F}_n-F)$ converges weakly in $L_1(\mu)$.
\end{proof}

\begin{lem} \label{lem:tight}
If $X_\alpha : \Omega_\alpha \to \mathbb{L}_\mu$ is asymptotically tight, it is asymptotically measurable if and only if $X_\alpha(t)$ is asymptotically measurable for every $t\in T$.
\end{lem}

\begin{lem} \label{lem:marginal}
If $X$ and $Y$ are tight Borel measurable maps into $\mathbb{L}_\mu$,
then $X$ and $Y$ are equal in law if and only if every marginal of $X$ and $Y$ is equal in law.
\end{lem}

\begin{proof}[Proofs of \cref{lem:tight,lem:marginal}]
These claims are not corollaries of \cite[Lemmas 1.5.2 and 1.5.3]{vw1996} since $C_b(\mathbb{L}_\mu)$ is bigger than $C_b(\mathbb{L}_T)$ and $C_b(\mathbb{L}_1)$, but they follow by the same logic.
\end{proof}

\begin{proof}[Proof of \cref{FfdAo8VYYjsJz}]
Necessity is immediate. We prove sufficiency.
If $X_\alpha$ is asymptotically tight and its marginals converge weakly, then $X_\alpha$ is asymptotically measurable by \cref{lem:tight}.
By Prohorov's theorem \cite[Theorem 1.3.9]{vw1996}, $X_\alpha$ is relatively compact.
Take any subnet in $X_\alpha$ that is convergent.
Its limit point is unique by \cref{lem:marginal} and the assumption that every marginal converges weakly.
Thus, $X_\alpha$ converges weakly.
The last statement is another consequence of Prohorov's theorem.
\end{proof}

\begin{proof}[Proof of \cref{dJMhMPoGhEIJvZ}]
We proceed (\ref{thmA6:2}) $\Rightarrow$ (\ref{thmA6:1}) $\Rightarrow$ (\ref{thmA6:3}) $\Rightarrow$ (\ref{thmA6:2}).

(\ref{thmA6:2}) $\Rightarrow$ (\ref{thmA6:1}). Fix $\varepsilon,\eta>0$.
Pick one $t_i$ from each $T_i$.
Then, $\|X_\alpha\|_T \leq \max_i |X_\alpha(t_i)|+\varepsilon$ with inner probability at least $1-\eta$.
Since the maximum of finitely many tight nets of real variables is tight and $\|X_\alpha\|_\mu$ is assumed to be tight, it follows that the net $\|X_\alpha\|_{\mathbb{L}_\mu}$ is asymptotically tight in $\mathbb{R}$.

Fix $\zeta>0$ and take $\varepsilon_m \searrow 0$. Let $M$ satisfy $\limsup P^\ast(\|X_\alpha\|_{\mathbb{L}_\mu}>M)<\zeta$. Taking $(\varepsilon,\eta)$ in (\ref{partition}) as $(\varepsilon_m,2^{-m}\zeta)$, we obtain for each $m$ a measurable partition $T=\bigcup_{i=1}^k T_i$ (suppressing dependence on $m$). For each $T_i$, enumerate all of the finitely many values $0=a_{i,0}\leq a_{i,1}\leq\cdots\leq a_{i,p}\leq M$ such that
\[
	\int_{T_i} (a_{i,j}-a_{i,j})|d\mu| \leq \frac{\varepsilon_m}{k} \quad \text{for} \quad j=1,\dots,p \quad \text{and} \quad
	\int_{T_i} a_{i,p}|d\mu| \leq M.
\]
Since $\mu$ is not necessarily finite on the whole $T$, on some partition $T_i$ the only choice of $a_{i,j}$ may be $0$.
Let $z_1,\dots,z_q$ be the finite exhaustion of all functions in $\mathbb{L}_\mu$ that are constant on each $T_i$ and take values on
\[
	0, \pm\varepsilon_m, \dots, \pm\lfloor M/\varepsilon_m \rfloor\varepsilon_m, \quad
	\pm a_{1,1}, \dots, \pm a_{1,p}, \quad
	\dots, \quad
	\pm a_{k,1}, \dots, \pm a_{k,p}.
\]
Let $K_m$ be the union of $q$ closed balls of radius $2\varepsilon_m$ around each $z_i$.
Then, since $\inf_j \int_{T_i} |X_\alpha-a_{i,j}||d\mu|\leq\frac{\varepsilon_m}{k}+\inf_x \int_{T_i} |X_\alpha-x||d\mu|$, the three conditions
\(
	\|X_\alpha\|_T \leq M,
\)
\(
	\sup_i \sup_{s,t\in T_i} |X_\alpha(s)-X_\alpha(t)| \leq \varepsilon_m,
\)
and
\(
	\sum_i \inf_x \int_{T_i} |X_\alpha-x| |d\mu| \leq \varepsilon_m
\)
imply that $X_\alpha\in K_m$. This holds for each $m$.

Let $K=\bigcap_{m=1}^\infty K_m$, which is closed, totally bounded, and therefore compact.
Moreover, we argue that for every $\delta>0$ there exists $m$ with $K^\delta \supset \bigcap_{j=1}^m K_j$. Suppose not.
Then there is a sequence $z_m$ not in $K^\delta$, but with $z_m \in \bigcap_{j=1}^m K_j$ for every $m$.
This has a subsequence contained in only one of the closed balls constituting $K_1$, and a further subsequence contained in only one of the balls constituting $K_2$, and so on.
The diagonal sequence of such subsequences would eventually be contained in a ball of radius $2\varepsilon_m$ for every $m$.
Therefore, it is Cauchy and its limit should be in $K$, which is a contradiction to the supposition $d(z_m,K)\geq\delta$ for every $m$.

Thus, if $X_\alpha$ is not in $K^\delta$, it is not in $\bigcap_{j=1}^m K_j$ for some $m$. Therefore,
\begin{multline*}
	P^\ast(X_\alpha \notin K^\delta) \leq P^\ast \Biggl( X_\alpha \notin \bigcap_{j=1}^m K_j \Biggr)
	\leq P^\ast(\|X_\alpha\|_{\mathbb{L}_\mu} > M) \\
	+ \sum_{j=1}^m P^\ast \biggl( \biggl[ \sup_i \sup_{s,t \in T_i} |X_\alpha(s) - X_\alpha(t)| \biggr] \vee \sum_i \inf_x \int_{T_i} |X_\alpha - x| |d\mu| > \varepsilon_j \biggr) \\
	\leq \zeta + \sum_{j=1}^m \zeta 2^{-j} < 2 \zeta.
\end{multline*}
Hence, we obtain $\limsup_\alpha P^\ast(X_\alpha \notin K^\delta) < 2 \zeta$, as asserted.

(\ref{thmA6:1}) $\Rightarrow$ (\ref{thmA6:3}).
If $X_\alpha$ is asymptotically tight, then so is each coordinate projection.
Therefore, $X_\alpha(t)$ is asymptotically tight in $\mathbb{R}$ for every $t\in T$.

Let $K_1 \subset K_2 \subset \cdots$ be a sequence of compact sets such that $\liminf P_\ast(X_\alpha \in K_m^\varepsilon) \geq 1-1/m$ for every $\varepsilon>0$. Define a semimetric $d$ on $T$ induced by $z$ by
\(
	d(s,t;z) \vcentcolon= |z(s) - z(t)|
	\vee \int_T |z| \mathbbm{1}\{z(s) \wedge z(t) \leq z \leq z(s) \vee z(t)\} \mathbbm{1}\{z(s) \neq z(t)\} |d\mu|.
\)
Observe that $d(s,s;z)=0$ and that $d$ is measurable with respect to $\mu$.%
\footnote{$T$ is not necessarily complete with respect to $d$.}
Now for every $m$, define a semimetric $\rho_m$ on $T$ by
\(
	\rho_m(s,t) \vcentcolon= \sup_{z\in K_m} d(s,t;z).
\)
We argue that $(T,\rho_m)$ is totally bounded. For $\eta>0$, cover $K_m$ by finitely many balls of radius $\eta$ centered at $z_1, \dots, z_k$.
Consider the partition of $\mathbb{R}^{2k}$ into cubes of edge length $\eta$.
For each cube, if there exists $t\in T$ such that the following $2k$\hyp{}tuple is in the cube,
\begin{multline*}
	r(t) \vcentcolon= \biggl( z_1(t), \ \int_T z_1 \mathbbm{1} \{0 \wedge z_1(t) \leq z_1 \leq 0 \vee z_1(t)\} |d\mu|, \quad \dots, \\
	z_k(t), \ \int_T z_k \mathbbm{1}\{0 \wedge z_k(t) \leq z_k \leq 0 \vee z_k(t)\} |d\mu| \biggr),
\end{multline*}
then pick one such $t$.
Since $\|z_j\|_{\mathbb{L}_\mu}$ is finite for every $j$ (i.e., the diameter of $T$ measured by each $d(\cdot,\cdot;z_j)$ is finite), this gives finitely many points $t_1,\dots,t_p$.
Notice that the balls $\{t : \rho_m(t,t_i)<3\eta\}$ cover $T$, that is, $t$ is in the ball around $t_i$ for which $r(t)$ and $r(t_i)$ are in the same cube; this follows because $\rho_m(t,t_i)$ can be bounded by
\(
	2 \sup_{z \in K_m} \inf_j \|z-z_j\|_{\mathbb{L}_\mu} +{} \sup_j d(t,t_i;z_j) < 3 \eta.
\)
The first term is the error of approximating $z(t)$ and $z(t_i)$ by $z_j(t)$ and $z_j(t_i)$; the second is the distance of $t$ and $t_i$ measured by $d(\cdot, \cdot; z_j)$.

Define the semimetric $\rho$ by
\(
	\rho(s,t) \vcentcolon= \sum_{m=1}^\infty 2^{-m} \bigl( \rho_m(s,t) \wedge 1 \bigr).
\)
We show that $(T,\rho)$ is still totally bounded. For $\eta>0$ take $m$ such that $2^{-m}<\eta$.
Since $T$ is totally bounded in $\rho_m$, we may cover $T$ with finitely many $\rho_m$\hyp{}balls of radius $\eta$.
Denote by $t_1,\dots,t_p$ the centers of such a cover.
Since $K_m$ is nested, we have $\rho_1 \leq \rho_2 \leq \cdots$.
Since we also have $\rho_m(t,t_i)<\eta$, for every $t$ there exists $t_i$ such that $\rho(t,t_i)\leq\sum_{k=1}^m 2^{-k}\rho_k(t,t_i)+2^{-m}<2\eta$.
Therefore, $(T,\rho)$ is totally bounded.

By definition we have $d(s,t;z)\leq\rho_m(s,t)$ for every $z\in K_m$ and $\rho_m(s,t)\wedge1\leq2^m\rho(s,t)$.
And if $\|z_0-z\|_{\mathbb{L}_\mu}<\varepsilon$ for $z\in K_m$, then $d(s,t;z_0)<2\varepsilon+d(s,t;z)$ for every pair $(s,t)$.
Hence, we conclude
\(
	K_m^\varepsilon \subset \bigl\{ z : \sup_{\rho(s,t)<2^{-m}\varepsilon}$ $d(s,t;z) \leq 3\varepsilon \bigr\}.
\)
Therefore, for $\delta<2^{-m}\varepsilon$,
\begin{multline*}
	\liminf_\alpha P_\ast \biggl( \sup_{\rho(s,t)<\delta} d(s,t;X_\alpha) \leq 3 \varepsilon \biggr) \\
	\begin{aligned}
	&\geq \liminf_\alpha P_\ast \biggl( \sup_{t \in T} \biggl[ \sup_{\rho(s,t)<\delta} |X_\alpha(s) - X_\alpha(t)| \vee \int_{0<\rho(s,t)<\delta} |X_\alpha(s)| |d\mu| \biggr] \leq 3 \varepsilon \biggr) \\
	&\geq 1 - \frac{1}{m}.
	\end{aligned}
\end{multline*}

(\ref{thmA6:3}) $\Rightarrow$ (\ref{thmA6:2}).
For $\varepsilon, \eta > 0$, take $\delta > 0$ as given.
Since $T$ is totally bounded, it can be covered with finitely many balls of radius $\delta$; let $t_1, \dots, t_K$ be their centers.
Disjointify the balls to obtain $\{T_i^\varepsilon\}$.
If $\int_{\{t_i\}} |X_\alpha| |d\mu| > 0$, then separate the partition $T_i^\varepsilon$ into $\{t_i\}$ and $T_i^\varepsilon \setminus \{t_i\}$.

There are three types of components in the partition: (a) singleton components (mass points) of $\mu$, (b) components with $|\mu|(T_i^\varepsilon) = \infty$, and (c) components with $|\mu|(T_i^\varepsilon) < \infty$.
The size of (a) is controlled by construction, so we control (b) and (c).
Clearly,
\begin{equation} \label{eq:LNSygY7HxTI29zAsT}
	\sup_{s,t \in T_i^{\varepsilon}} |X_\alpha(s)-X_\alpha(t)| \leq 2 \sup_{\rho(s,t_i)<\delta} |X_\alpha(s)-X_\alpha(t_i)| \leq 2 \varepsilon.
\end{equation}
Denote by $i_\infty$ the index for which $|\mu|(T_{i_\infty}^\varepsilon) = \infty$.
Now we argue that $\sum_{i_\infty} \int_{T_{i_\infty}^\varepsilon} |X_\alpha| |d\mu|$ can be arbitrarily small (with inner probability at least $1-\eta$) for sufficiently small $\varepsilon$.
By construction, $\sup_{s \in T_{i_\infty}^\varepsilon} |X_\alpha(s)| \leq 2\varepsilon$.%
\footnote{This follows because $\inf_{T_{i_\infty}^\varepsilon} |X_\alpha| = 0$ given that $\int_{T_{i_\infty}^\varepsilon} |X_\alpha| |d\mu| < \infty$.}
Thus, $\sum_{i_\infty} \int_{T_{i_\infty}^\varepsilon} |X_\alpha| |d\mu| \leq \int_T |X_\alpha| \mathbbm{1}\{|X_\alpha| \leq 2\varepsilon\} |d\mu|$.
Since $T$ is totally bounded, $\int_T |X_\alpha| |d\mu|$ is bounded by $K\varepsilon$ with inner probability at least $1-\eta$ (proving asymptotic tightness of $\|X_\alpha\|_\mu$), and hence the previous integral must be arbitrarily small for small $\varepsilon$.
Now we turn to (c).
Let $\varepsilon'$ be such that
\begin{equation} \label{eq:HvTlSWEYWxlmTaN1i86}
	\limsup_\alpha P^\ast \biggl( \int_T |X_\alpha| \mathbbm{1}\{|X_\alpha| \leq 3 \varepsilon'\} |d\mu| > \varepsilon \biggr) < 1 - \eta.
\end{equation}
For each $T_i^\varepsilon$ with $|\mu|(T_i^\varepsilon) < \infty$, construct a further partition of it $\{T_j^{\varepsilon'}\}$ with this $\varepsilon'$.
Note that $\{T_{i_\infty}^\varepsilon\} \cup \{T_j^{\varepsilon'}\}$ defines another finite partition of $T$.
If there exists $s \in T_j^{\varepsilon'}$ such that $|X_\alpha(s)| \leq \varepsilon'$, then by construction $\sup_{t \in T_j^{\varepsilon'}} |X_\alpha(t)| \leq 3 \varepsilon'$.
The contrapositive of this is also true.
Thus, observing
\(
	\sum_j \inf_x \int_{T_j^{\varepsilon'}} |X_\alpha - x| |d\mu|
	\leq \sum_j \inf_x \int_{T_j^{\varepsilon'}} |X_\alpha - x| \mathbbm{1}\{|X_\alpha| > \varepsilon'\} |d\mu| + \int_T |X_\alpha| \mathbbm{1}\{|X_\alpha| \leq 3 \varepsilon'\} |d\mu|,
\)
we may assume $\inf_{T'} |X_\alpha(s)| \geq \varepsilon' > 0$ at the cost of one more $\varepsilon$.
Then, we also have $\int_{T'} |d\mu| \leq K\varepsilon/\varepsilon'$ since $\varepsilon' \int_{T'} |d\mu| \leq \int_T |X_\alpha| |d\mu|$.
For the partition $T_j^{\varepsilon'}$ of $T'$, further construct a nested finite partition $T_k^{\varepsilon'/K}$. Now
\begin{multline*}
	\sum_k \inf_x \int_{T_k^{\varepsilon'/K}} |X_\alpha - x| |d\mu| \\
	\leq \sum_k \sup_{s, t \in T_k^{\varepsilon'/K}} |X_\alpha(s) - X_\alpha(t)| \int_{T_k^{\varepsilon'/K}} |d\mu|
	\leq \frac{\varepsilon'}{K} \int_{T'} |d\mu| \leq \varepsilon
\end{multline*}
with inner probability at least $1-\eta$.
This, (\ref{eq:LNSygY7HxTI29zAsT}), and (\ref{eq:HvTlSWEYWxlmTaN1i86}) yield the result.
%
\end{proof}

\begin{proof}[Proof of \cref{X4cQDt7zeDM7E9hTYSZ}]
We first work on the case $r<1$.
Define
\(
	\rho(s,t) \vcentcolon= \int_{(s,t)} u^r(1-u)^r dQ.
\)
We show that $(0,1)$ is totally bounded with respect to $\rho$.
Observe that \cref{lem:t97410h} (\ref{lemA1:5}) and $r>\frac{1}{2+c}$ imply $u^r(1-u)^r Q(u) \to 0$ as $u\to\{0,1\}$.
Therefore, integrating by parts,
\[
	\rho(0,1) \leq \int_{(0,1)} u^r \wedge(1-u)^r dQ \leq |Q|\Bigl(\frac{1}{2}\Bigr) + \int_0^{\frac{1}{2}} u^{r-1} |Q| du + \int_{\frac{1}{2}}^1 (1-u)^{r-1} |Q| du.
\]
Since $Q \in L_{2+c}$ and $u^{r-1}\wedge(1-u)^{r-1} \in L_q$ for every $q<1/(1-r)$, in particular for $q=(2+c)/(1+c)$, this integral is finite by H\"older's inequality.
This means the diameter of $(0,1)$ is finite, so $(0,1)$ is totally bounded.

Note that $|Q|$ is eventually smaller than $u^{-r}(1-u)^{-r}$ near $0$ and $1$, so that for every $\eta>0$ there exists $M$ such that
\[
	\limsup_\alpha P^\ast(\|(|Q|\vee1)X_\alpha\|_\infty>M)\leq\limsup_\alpha P^\ast\biggl(\biggl\|\frac{X_\alpha}{u^r(1-u)^r}\biggr\|_\infty>M\biggr)<\eta.
\]
This shows uniform equicontinuity.
Next, for every $0<s\leq t<1$,
\[
	\int_{(s,t)}|X_\alpha|dQ\leq\sup_{u\in(0,1)}\frac{|X_\alpha(u)|}{u^r(1-u)^r}\int_{(s,t)}v^r(1-v)^r dQ
	\leq\biggl\|\frac{X_\alpha}{u^r(1-u)^r}\biggr\|_\infty\rho(s,t).
\]
Therefore,
\[
	P^\ast \biggl( \sup_{t\in(0,1)} \int_{0<\rho(s,t)<\delta} |X_\alpha| dQ(s) > \varepsilon \biggr)
	\leq P^\ast \biggl( \biggl\| \frac{X_\alpha}{u^r(1-u)^r} \biggr\|_\infty > \frac{\varepsilon}{\delta} \biggr).
\]
By assumption, this can be however small by the choice of $\delta$.
Conclude that $X_\alpha$ is asymptotically $(\rho,Q)$\hyp{}equiintegrable in probability.

Finally, for $r \geq 1$, replace every $r$ by $1/2$. Then the result follows since $\bigl\| \frac{X_\alpha}{u^r(1-r)^r} \bigr\|_\infty \geq \bigl\| \frac{X_\alpha}{u^{1/2}(1-u)^{1/2}} \bigr\|_\infty$.
\end{proof}

\begin{proof}[Proof of \cref{p34AI8lskauT}]
Assume without loss of generality $M=1$.
Define $U_{(0)} \vcentcolon= 0$. Let $\tilde{\mathbb{F}}_n$ and $\tilde{\mathbb{G}}_n$ be the continuous linear interpolations of $\mathbb{F}_n$ and $\mathbb{G}_n$, that is, for $U_{(i-1)} \leq u < U_{(i)}$,
\(
	\tilde{\mathbb{F}}_n(u) \vcentcolon= \frac{i-1}{n} + \frac{u-U_{(i-1)}}{n(U_{(i)}-U_{(i-1)})}
\)
and
\(
	\tilde{\mathbb{G}}_n(u) \vcentcolon= \frac{1}{n} \sum_{i=1}^n w_{i,n} \mathbbm{1}\{U_i\leq u\} + \frac{w_{i,n}(u-U_{(i-1)})}{n(U_{(i)}-U_{(i-1)})},
\)
and for $u\geq U_{(n)}$, $\tilde{\mathbb{F}}_n(u) \vcentcolon= 1$ and $\tilde{\mathbb{G}}_n(u) \vcentcolon= \frac{1}{n} \sum w_{i,n}$.
Observe that $\mathbb{K}_n(u) = \tilde{\mathbb{G}}_n(\tilde{\mathbb{F}}_n^{-1}(u))$.
By \cref{bZivHihLTt5o} it suffices to show that $\sqrt{n}(\tilde{\mathbb{F}}_n-I)$ and $\sqrt{n}(\tilde{\mathbb{G}}_n-K)$ converge weakly jointly in $\mathbb{L}_Q$.
Note that
\(
	\|\mathbb{F}_n-I\|_\infty - \frac{1}{n} \leq \| \tilde{\mathbb{F}}_n-I \|_\infty \leq \|\mathbb{F}_n-I\|_\infty + \frac{1}{n}
\)
and
\(
	\|\mathbb{F}_n-I\|_Q - C \leq \|\tilde{\mathbb{F}}_n-I\|_Q \leq \|\mathbb{F}_n-I\|_Q + C
\)
for $C\vcentcolon=\int (\tilde{I}-\lfloor n\tilde{I}\rfloor/n) dQ=O(1/n)$.
Thus, $\sqrt{n}(\tilde{\mathbb{F}}_n-I)$ converges weakly in $\mathbb{L}_Q$ if and only if $\sqrt{n}(\mathbb{F}_n-I)$ does, and they share the same limit.
The same is true for $\tilde{\mathbb{G}}_n$ and $\mathbb{G}_n$.

The classical results imply that $\sqrt{n}(\mathbb{F}_n-I)$ converges weakly in $L_\infty$ to a Brownian bridge and
\(
	\bigl\| \frac{\sqrt{n}(\mathbb{F}_n-I)}{u^r(1-u)^r} \bigr\|_\infty = O_P(1)
\)
for every $r<1/2$ \cite{ch1993}.
By \cref{X4cQDt7zeDM7E9hTYSZ,dJMhMPoGhEIJvZ}, it follows that $\sqrt{n}(\mathbb{F}_n-I)$ converges weakly in $\mathbb{L}_Q$.
By assumption $\sqrt{n}(\mathbb{G}_n-K)$ converges weakly in $L_\infty$ jointly with $\sqrt{n}(\mathbb{F}_n-I)$, and since $M|\mathbb{F}_n-I|\geq|\mathbb{G}_n-K|$, conclude that $\sqrt{n}(\mathbb{G}_n-K)$ converges weakly in $\mathbb{L}_Q$ jointly with $\sqrt{n}(\mathbb{F}_n-I)$.
\end{proof}

\begin{lem}[Inverse composition map] \label{bZivHihLTt5o}
Let $\mathbb{L}_Q$ contain the identity map $I(u)\vcentcolon=u$.
Let $\mathcal{D}$ be the subset of $\mathbb{L}_Q \times \mathbb{L}_Q$ such that every $(A,B) \in \mathcal{D}$ satisfies $A(u_1)-A(u_2)\geq B(u_1)-B(u_2)\geq0$ for every $u_1\geq u_2$, the range of $A$ contains $(0,1)$, and $B$ is differentiable and Lipschitz.
Let $\mathbb{L}_{Q,\textrm{UC}}$ be the subset of $\mathbb{L}_Q$ of uniformly continuous functions.
Then, the map $\chi : \mathcal{D} \to \mathbb{L}_Q$, $\chi(A,B) \vcentcolon= B\circ A^{-1}$, is Hadamard differentiable at $(A,B) \in \mathcal{D}$ for $A=I$ tangentially to $\mathbb{L}_Q \times \mathbb{L}_{Q,\textrm{UC}}$. The derivative is given by
\(
	\chi_{I,B}'(a,b)(u) = b(u) + B'(u) a(u)$ for $u \in (0,1).
\)
\end{lem}

\begin{proof}
For $(A,B)\in\mathcal{D}$ and $u_1 \geq u_2$, denote $v_1\vcentcolon=A(u_1)$ and $v_2\vcentcolon=A(u_2)$.
By assumption we have $v_1-v_2\geq B(A^{-1}(v_1))-B(A^{-1}(v_2))\geq0$ for every $v_1\geq v_2$.
Therefore, $B\circ A^{-1}$ is monotone and bounded by the identity map up to a constant.
This implies
\(
	\int_{(0,1)} \bigl| \widetilde{B \circ A^{-1}} \bigr| dQ \leq \int_{(0,1)} |\tilde{I}| dQ < \infty
\)
and $\|Q(B\circ A^{-1})\|_\infty<\infty$; it follows that $B \circ A^{-1}$ is in $\mathbb{L}_{Q,1}$.

Let $a_t \to a$ and $b_t \to b$ in $\mathbb{L}_Q$ and $(A_t,B_t)\vcentcolon=(I+ta_t,B+tb_t) \in \mathcal{D}$.
We want to show
\(
	\bigl\| \frac{B_t\circ A_t^{-1}-B\circ I^{-1}}{t} - b - B' a \bigr\|_{\mathbb{L}_Q} \conv 0
\)
as $t\to0$.
That $\|\cdot\|_{Q,\infty} \to 0$ follows by applying \cite[Lemma 3.9.27]{vw1996} to $(A^{-1},QB)$ as elements in $L_\infty$.
Thus, it remains to show $\|\cdot\|_Q \to 0$.
In the assumed inequality, substitute $(u_1,u_2)$ by $(u,A_t^{-1}(u))$ to find that
\(
	|A_t(u)-u| \geq |B_t(A_t^{-1}(u))-B_t(u)| \geq 0.
\)
Therefore, the following inequality holds pointwise:
\(
	|B_t\circ A_t^{-1}-B| \leq |B_t\circ A_t^{-1}-B_t|+|B_t-B| \leq |A_t-I|+|B_t-B|=|ta_t|+|tb_t|.
\)
For $\varepsilon>0$, write $\|\cdot\|_Q$ as
\(
	\bigl( \int_0^{\varepsilon} + \int_\varepsilon^{1-\varepsilon} + \int_{1-\varepsilon}^1 \bigr) \bigl| \frac{B_t\circ A_t^{-1}-B}{t} - b - B' a \bigr| dQ.
\)
For any fixed $\varepsilon>0$ the middle term vanishes as $t\to0$ since $\|\cdot\|_{Q,\infty}\to0$.
It remains to show that the first term can be arbitrarily small since then by symmetry the third term is also ignorable.
Using the above inequality, write
\(
	\int_0^{\varepsilon} \bigl| \frac{B_t\circ A_t^{-1}-B}{t} - b - B' a \bigr| dQ \leq \int_0^{\varepsilon} (|a_t|+|b_t|+|b|+|B'a|) dQ.
\)
Since $\|a_t-a\|_Q\to0$ and $\|b_t-b\|_Q\to0$, this integral can be arbitrarily small by the choice of $\varepsilon$, as desired.
\end{proof}


\subsection{Convergence of Quantile Processes as Integrable Processes}

If $m$ is an identity, we denote $\mathbb{L}_m$ and $\mathbb{L}_{m,\phi}$ by $\mathbb{L}$ and $\mathbb{L}_\phi$, and $\|\cdot\|_m$ by $\|\cdot\|_1$.

We first establish differentiability of the inverse map for distribution functions with finite first moments.

\begin{lem}[Inverse map] \label{2v7DhDjurrWYxBof}
Let $F \in \mathbb{L}_\phi$ be a distribution function on (an interval of) $\mathbb{R}$ that has at most finitely many jumps and is otherwise continuously differentiable with strictly positive density $f$.
Then, the inverse map $\phi : \mathbb{L}_\phi \to \mathbb{B}$, $\phi(F)\vcentcolon=Q=F^{-1}$, is Hadamard differentiable at $F$ tangentially to the set $\mathbb{L}_0$ of all continuous functions in $\mathbb{L}$. The derivative is
\(
	\phi'_F(z) = -(z\circ Q) Q'. 
\)
\end{lem}


\begin{proof}[Proof of \cref{2v7DhDjurrWYxBof}]
Take $z_t \to z$ in $\mathbb{L}$ and $F_t \vcentcolon= F+tz_t \in \mathbb{L}_\phi$.
We want to show
\(
	\bigl\| \frac{\phi(F_t)-\phi(F)}{t}-\phi'_F(z) \bigr\|_{\mathbb{B}}=\int_0^1 \bigl| \frac{\phi(F_t)-\phi(F)}{t}-\phi'_F(z) \bigr| du \conv 0
\)
as $t \to 0$.
Let $j\in\mathbb{R}$ be a point of jump of $F$.
For small $\varepsilon>0$, split the integral as
\[
	\biggl( \int_0^{F(j-\varepsilon)} + \int_{F(j-\varepsilon)}^{F(j+\varepsilon)} + \int_{F(j+\varepsilon)}^1 \biggr) \biggl| \frac{\phi(F_t)-\phi(F)}{t}-\phi'_F(z) \biggr| du.
\]
Observe that the second integral is bounded by
\[
	2\varepsilon\biggl\|\frac{F_t-F}{t}\biggr\|_\infty + \biggl( \int_{F(j-\varepsilon)}^{F(j-)} + \int_{F(j)}^{F(j+\varepsilon)} \biggr) |\phi_F'(z)|du.
\]
The first term of this equals $2\varepsilon\|z_t\|_\infty$ and can be arbitrarily small by the choice of $\varepsilon$.
If $\varepsilon$ is small enough that there is no other jump in $[j-\varepsilon,j+\varepsilon]$, by Fubini's theorem,
\(
	\int_{F(j-\varepsilon)}^{F(j-)}|\phi_F'(z)|du = \int_{j-\varepsilon}^{j} \bigl|\frac{z}{f}\bigr|dF \leq \varepsilon\|z\|_\infty,
\)
which can be, again, arbitrarily small.
Similarly for the last integral.
Therefore, we ignore finitely many jumps of $F$ so that $f>0$ everywhere.

For $\varepsilon>0$ there exists $M$ such that $F(-M) < \varepsilon$ and $1-F(M) < \varepsilon$.
Write
\begin{multline*}
	\biggl\| \frac{\phi(F_t)-\phi(F)}{t}-\phi'_F(z) \biggr\|_{\mathbb{B}}
	\leq \int_{F(-M)+\varepsilon}^{F(M)-\varepsilon} \biggl| \frac{\phi(F_t)-\phi(F)}{t}-\phi'_F(z) \biggr| du \\
	+ \biggl( \int_0^{2\varepsilon} + \int_{1-2\varepsilon}^1 \biggr) \biggl| \frac{\phi(F_t)-\phi(F)}{t}-\phi'_F(z) \biggr| du.
\end{multline*}
By \cite[Theorem 3.9.23 (i)]{vw1996}, the integrand vanishes uniformly on $[F(-M)+\varepsilon,F(M)-\varepsilon]$.
As the first integral is bounded by
\(
	\sup_{u\in[F(-M)+\varepsilon,F(M)-\varepsilon]}$ $\bigl|\frac{\phi(F_t)-\phi(F)}{t}-\phi_F'(z)\bigr|,
\)
it vanishes as $t\to0$.
Now turn to the second integral. The triangle inequality bounds it by
\(
	\int_0^{2\varepsilon} \bigl| \frac{\phi(F_t)-\phi(F)}{t} \bigr| du + \int_0^{2\varepsilon} |\phi'_F(z)| du.
\)
Since $F$ and $F_t$ are nondecreasing, by Fubini's theorem,
\begin{multline*}
	\int_0^{2\varepsilon} \biggl| \frac{\phi(F_t)-\phi(F)}{t} \biggr| du = \frac{1}{|t|} \int_0^{2\varepsilon} \bigl| F_t^{-1}-F^{-1} \bigr| du \\
	\leq \frac{1}{|t|} \int_{-\infty}^{F^{-1}(2\varepsilon+\|tz_t\|_\infty)} |tz_t| dx \leq \|z_t-z\|_1 + \int_{-\infty}^{F^{-1}(2\varepsilon+t\|z_t\|_\infty)} |z| dx.
\end{multline*}
The first term goes to $0$ and the second term can be arbitrarily small by the choice of $\varepsilon$.
Finally, by the change of variables,
\(
	\int_0^{2\varepsilon} |\phi'_F(z)| du = \int_{-\infty}^{F^{-1}(2\varepsilon)} \bigl| \frac{z}{f} \bigr| dF = \int_{-\infty}^{F^{-1}(2\varepsilon)} |z| dx,
\)
which can be arbitrarily small.
Likewise, the integral from $1-2\varepsilon$ to $1$ converges to $0$.
This completes the proof.
\end{proof}

Now we allow transformations of locally bounded variation.
A function of locally bounded variation admits decomposition into the difference of two monotone functions.
Then, we exploit the relationship $m(F^{-1}) = (F \circ m^{-1})^{-1}$ for a monotone $m$ and use the chain rule.

\begin{proof}[Proof of \cref{XCNkYqAFjk2IK6F}]
Since $m$ is of locally bounded variation, write $m(x) = m_1(x)-m_2(x)$ where $m_1$ and $m_2$ are increasing.
Moreover, $m_1$ and $m_2$ can be chosen to be continuously differentiable and strictly increasing, and for their corresponding Lebesgue\hyp{}Stieltjes measures $\mu_1$ and $\mu_2$, $F$ belongs to both $\mathbb{L}_{\mu_1,\phi}$ and $\mathbb{L}_{\mu_2,\phi}$.
Since the derivative formula is linear in $m'$, it suffices to show the claim for $m_1$ and $m_2$ separately.
%
Now observe that $z$ is in $\mathbb{L}_\mu$ (or $\mathbb{L}_{\mu,0}$) if and only if $z \circ m^{-1}$ is in $\mathbb{L}$ (or $\mathbb{L}_0$).
The assertion then follows by \cite[Lemma 3.9.3]{vw1996} applied to \cref{2v7DhDjurrWYxBof,fB7dW2UeCqyz3bdNZp}.
\end{proof}

\begin{lem} \label{fB7dW2UeCqyz3bdNZp}
Let $m : \mathbb{R} \to \mathbb{R}$ be a strictly increasing continuous function and $\mu$ be the associated Lebesgue\hyp{}Stieltjes measure.
Then, the map $\psi : \mathbb{L}_\mu \to \mathbb{L}$, $\psi(F) \vcentcolon= F \circ m^{-1}$, is 
uniformly Fr\'echet differentiable with rate function $q \equiv 0$.%
\footnote{A map $\psi : \mathbb{L} \to \mathbb{B}$ is {\em uniformly Fr\'echet differentiable with rate function $q$} if there exists a continuous linear map $\psi_F' : \mathbb{L} \to \mathbb{B}$ such that $\|\psi(F+z)-\psi(F)-\psi_F'(z)\|_{\mathbb{B}}=O(q(\|z\|_{\mathbb{L}}))$ uniformly over $F\in\mathbb{L}$ as $z \to 0$ and $q$ is monotone with $q(t)=o(t)$.}
The derivative is given by $\psi_F'(z) \vcentcolon= z \circ m^{-1}$.
\end{lem}

\begin{proof}[Proof of \cref{fB7dW2UeCqyz3bdNZp}]
Observe that $\psi(F+z) - \psi(F) = (F+z)(m^{-1}) - F(m^{-1}) = z(m^{-1})$.
Therefore, $\psi(F+z)-\psi(F)-\phi_F'(z) = 0$.
\end{proof}

\begin{proof}[Proof of \cref{dqt032FkNegj}]
This follows from \cref{prop:A:F,XCNkYqAFjk2IK6F}.
\end{proof}


\subsection{Convergence of $L$\hyp{}statistics}

\begin{proof}[Proof of \cref{thm:wilcoxon}]
The derivative map is linear by construction; it is also continuous since
\(
	|\lambda_{Q,K}'(z_1,\kappa_1)-\lambda_{Q,K}'(z_2,\kappa_2)|=\bigl|\int_0^1Qd(\kappa_1-\kappa_2)+\int_0^1(z_1-z_2)dK\bigr|
	\leq\|\kappa_1-\kappa_2\|_{Q,\infty}+\|\kappa_1-\kappa_2\|_{Q}+M\|z_1-z_2\|_{\mathbb{B}},
\)
which vanishes as $\|z_1-z_2\|_{\mathbb{B}}\to0$ and $\|\kappa_1-\kappa_2\|_{\mathbb{L}_Q}\to0$.
Let $z_t\to z$ and $\kappa_t\to\kappa$ such that $Q_t\vcentcolon=Q+tz_t$ is in $\mathbb{B}$ and $K_t\vcentcolon=K+t\kappa_t$ is in $\mathbb{L}_{Q,M}$. Observe
\[
	\frac{\lambda(Q_t,K_t)-\lambda(Q,K)}{t} - \lambda'_{Q,K}(z_t,\kappa_t) = \int (z_t-z) d(K_t-K) + \int z d(K_t-K).
\]
The first term vanishes since
\(
	\bigl| \int (z_t-z) d(K_t-K) \bigr| \leq 2M \int |z_t-z| du = 2M \|z_t-z\|_{\mathbb{B}}.
\)
As $z$ is integrable, for every $\varepsilon>0$ there exists a small number $\delta>0$ such that
\(
	\bigl(\int_0^\delta+\int_{1-\delta}^1\bigr) |z| du + \int_\delta^{1-\delta} (|z| - (|z| \wedge \delta^{-1})) du \leq \varepsilon.
\)
This gives
\begin{align*}
	\biggl| \int z d(K_t-K) \biggr| &\leq \biggl| \int_\delta^{1-\delta} (-\delta^{-1} \vee z \wedge \delta^{-1}) d(K_t-K) \biggr| \\
	&\hphantom{={}} + \biggl| \int z d(K_t-K) - \int_\delta^{1-\delta} (-\delta^{-1} \vee z \wedge \delta^{-1}) d(K_t-K) \biggr| \\
	&\leq \biggl| \int_\delta^{1-\delta} (-\delta^{-1} \vee z \wedge \delta^{-1}) d(K_t-K) \biggr| + 2M\varepsilon.
\end{align*}
Let $\tilde{z} \vcentcolon= -\delta^{-1} \vee z \wedge \delta^{-1}$.
Since $\tilde{z}$ is ladcag on $[\delta,1-\delta]$, there exists a partition $\delta=t_0<t_1<\cdots<t_m=1-\delta$ such that $\tilde{z}$ varies less than $\varepsilon$ on each interval $(t_{i-1},t_i]$.
Let $\bar{z}$ be the piecewise constant function that equals $\tilde{z}(t_i)$ on each interval $(t_{i-1},t_i]$. Then
\begin{multline*}
	\biggl| \int_\delta^{1-\delta} \tilde{z} d(K_t-K) \biggr| \leq 2M \sup_{u\in[\delta,1-\delta]}|\tilde{z}-\bar{z}| + |\tilde{z}(\delta)| |(K_t-K)(\{\delta\})| \\
	+ \sum_{i=1}^m |\tilde{z}(t_i)| |(K_t-K)((t_{i-1},t_i])|.
\end{multline*}
The first term is arbitrarily small by the choice of $\varepsilon$, and the second and third terms are collectively bounded by $(2m+1)\delta^{-1}\|K_t-K\|_\infty=(2m+1)\delta^{-1}t\|\kappa_t\|_\infty$, which converges to $0$ regardless of the choice of $K$.

The proof for $\tilde{\lambda}$ is basically the same.
\end{proof}

\begin{proof}[Proof of \cref{EkaRgpsiHieiNgk4NFP}]
Weak convergence follows from \cref{dqt032FkNegj,p34AI8lskauT,thm:wilcoxon}.
The derivative formulas give us
\begin{align*}
	&\Cov(\xi_j, \xi_k) = \int_0^1 \int_0^1 (m_j'\circ Q_j)Q_j'(s)(m_k'\circ Q_k)Q_k'(t) [F_{ik}^Q(s,t)-st] ds dt \\
	&+ \int_0^1 \int_0^1 (m_j'\circ Q_j)Q_j'(s)(m_k'\circ Q_k)Q_k'(t) [K_{jk}F_{jk}^Q(s,t)-stK_j(s)K_k(t)] ds dt \\
	&- \int_0^1 \int_0^1 (m_j'\circ Q_j)Q_j'(s)(m_k'\circ Q_k)Q_k'(t) K_j(s) [F_{jk}^Q(s,t)-st] ds dt \\
	&- \int_0^1 \int_0^1 )m_j'\circ Q_j)Q_j'(s)(m_k'\circ Q_k)Q_k'(t) K_k(t) [F_{jk}^Q(s,t)-st] ds dt.
\end{align*}
Consistency of the sample analogue estimator follows from uniform convergence of $\mathbb{K}_{n,j}^F$ and $\mathbb{K}_{n,k}^F$ and \cref{add:1}.
\end{proof}

\begin{lem} \label{add:1}
Let $m : \mathbb{R} \to \mathbb{R}$ be a ladcag increasing function. For a probability measure $F$ on $\mathbb{R}$ such that $\mathbb{E}[m(X)]<\infty$, $X\sim F$, we have
\begin{gather*}
	\bigl\| m(t) \mathbb{F}_n(t) - m(t) F(t) \bigr\|_\infty \conv^{\oas} 0, \quad
	\biggl\| \int_{[s,t]} |m| d\mathbb{F}_n - \int_{[s,t]} |m| dF \biggr\|_\infty \conv^{\oas} 0, \\
	\biggl\| \int_{[s,t]} |\tilde{\mathbb{F}}_n| dm - \int_{[s,t]} |\tilde{F}| dm \biggr\|_\infty \conv^{\as} 0, \qquad
	\int_{\mathbb{R}} |\mathbb{F}_n - F| dm \conv^{\as} 0,
\end{gather*}
where the suprema are each taken over $t \in \mathbb{R}$, $(s,t) \in \overline{\mathbb{R}}^2$, and $(s,t) \in \overline{\mathbb{R}}^2$.
\end{lem}

\begin{proof}[Proof of \cref{add:1}]
We assume $m(0)=0$ without loss of generality.
In view of \cite[Theorem 2.4.1]{vw1996}, the first two claims follow if
\begin{gather*}
	\mathcal{F} = \bigl\{ f_t : \mathbb{R} \to \mathbb{R} : t \in \overline{\mathbb{R}}, \, f_t(x) = m(t) \mathbbm{1}\{x \leq t\} \bigr\}, \\
	\mathcal{G} = \bigl\{ g_{s,t} : \mathbb{R} \to \mathbb{R} : s, t \in \overline{\mathbb{R}}, \, g_{s,t}(x) = |m(x)| \mathbbm{1}\{s \leq x \leq t\} \bigr\}
\end{gather*}
have finite bracketing numbers with respect to $L_1(P)$. 
For $\mathcal{F}$ take $-\infty = t_0 < t_1 < \cdots < t_m = \infty$ such that $|\int (f_{t_{i+1}} - f_{t_i}) dF| < \varepsilon$ for each $i$ and consider the brackets $\{f_{t_i}\}$.%
\footnote{If $F$ has a probability mass at $t$, then for small $\varepsilon$ take, instead of $f_t$, $\tilde{f}_{t,c}(x) = m(t) [c \mathbbm{1}\{x \leq t\} + (1-c) \mathbbm{1}\{x < t\}]$ for appropriately chosen $c$.}
This partition is finite by \cref{lem:t97410h} and $\mathbb{E}[m(X)]<\infty$.
For $\mathcal{G}$ take $-\infty = t_0 < t_1 < \cdots < t_m = \infty$ such that $\bigl| \int_{(-\infty,t_{i+1}]} |m| dF - \int_{(-\infty,t_i]} |m| dF \bigr| < \varepsilon$ for each $i$ and consider the brackets $\{g_{s,t}\}$ for every pair $s, t \in \{t_0, \dots, t_m\}$.%
\footnote{Again, if $F$ has a mass, similar adjustments are needed.}
This partition is finite by $\mathbb{E}[m(X)]<\infty$.

For the third claim, observe that
\(
	\int_{[s,t]} |m| d\mathbb{F}_n = \int_{[s,t]} |m| d\tilde{\mathbb{F}}_n
	= \bigl[ |m| \tilde{\mathbb{F}}_n \bigr]_s^t + \int_{[s,t]} |\tilde{\mathbb{F}}_n| d\mu.
\)
Then the claim follows by the first two claims and the triangle inequality,
\(
	\bigl\| \int_{[s,t]} |\tilde{\mathbb{F}}_n| d\mu - \int_{[s,t]} |\tilde{F}| d\mu \bigr\|_\infty
	\leq 2 \bigl\| m(t) \mathbb{F}_n(t) - m(t) F(t) \bigr\|_\infty
	+ \bigl\| \int_{[s,t]} |m| d\mathbb{F}_n - \int_{[s,t]} |m| dF \bigr\|_\infty.
\)%
\footnote{Measurability of the sup on the LHS follows by the continuity of Lebesgue integrals.}

For the last claim, observe that \cref{lem:t97410h} and the preceding claim imply that for $\varepsilon > 0$ there exists $M < \infty$ such that
\(
	\bigl( \int_{(-\infty,-M]} + \int_{[M,\infty)} \bigr) |\tilde{\mathbb{F}}_n| d\mu + \bigl( \int_{(-\infty,-M]} + \int_{[M,\infty)} \bigr) |\tilde{F}| d\mu < \varepsilon
\)
with probability tending to $1$. By the triangle inequality,
\(
	\int_{\mathbb{R}} \bigl| \tilde{\mathbb{F}}_n - \tilde{F} \bigr| d\mu \leq \int_{(-M,M)} \bigl| \tilde{\mathbb{F}}_n - \tilde{F} \bigr| d\mu + \varepsilon
	\leq \|\mathbb{F}_n - F\|_\infty \mu((-M,M)) + \varepsilon.
\)
Then the assertion follows by the Glivenko\hyp{}Cantelli theorem.
\end{proof}


\subsection{Validity of Nonparametric Bootstrap} \label{sec:proofs:bootstrap}

We start with the key lemma in Poissonization, the counterpart of \cite[Lemma 3.6.16]{vw1996}.

\begin{lem} \label{lem:A:Poissonization}
For each $n$, let $(W_{n1},\dots,W_{nn})$ be an exchangeable nonnegative random vector independent of $X_1,X_2,\dots$ such that $\sum_{i=1}^nW_{ni}=1$ and $\max_{1\leq i\leq n}|W_{ni}|$ converges to zero in probability.
Let $F$ be a probability distribution on $\mathbb{R}$ such that $\int_{\mathbb{R}}\sqrt{F(1-F)}|d\mu|<\infty$.
Then, for every $\varepsilon>0$, as $n\to\infty$,
\(
	{\Pr}_W\bigl(\bigl\|\sum_{i=1}^nW_{ni}\bigl(\mathbbm{1}\{X_i\leq x\}-F(x)\bigr)\bigr\|_\mu^\ast>\varepsilon\bigr)\conv^{\oas}0.
\)
\end{lem}


\begin{proof}[Proof of \cref{lem:A:Poissonization}]
Assume without loss of generality that $\mu$ is a positive measure and let $m(x)\vcentcolon=\mu([0,x))$ for $x\geq0$ and $\mu([x,0))$ for $x<0$.
Since \cite[Lemma 3.6.7]{vw1996} goes through with $\|\cdot\|_{\mathbb{L}_\mu}$, the proof of this lemma is almost identical to \cite[Lemma 3.6.16]{vw1996}.
Essentially, the only part that requires modification is boundedness of $n^{-1}\sum_{i=1}^n\|\mathbbm{1}\{X_i\leq x\}-F(x)\|_\mu^r$ ($r<1$).
Note that
\(
	|\mathbbm{1}\{X_i\leq x\}-F(x)|\leq|\mathbbm{1}\{X_i\leq x\}-\mathbbm{1}\{0\leq x\}|+|\tilde{F}(x)|.
\)
Therefore,
\(
	\|\mathbbm{1}\{X_i\leq x\}-F(x)\|_\mu 
	\leq m(X_i)+\|\tilde{F}\|_\mu.
\)
Find that
\(
	\frac{1}{n}\sum_{i=1}^n\|\mathbbm{1}\{X_i\leq x\}-F(x)\|_\mu^r\leq\frac{1}{n}\sum_{i=1}^n m(X_i)^r+\|\tilde{F}\|_\mu^r,
\)
which converges almost surely to $\mathbb{E}[m(X_i)^r]+\|\tilde{F}\|_\mu^r<\infty$.
\end{proof}

Given this, we infer as in \cite[Theorem 3.6.1]{vw1996} that conditional weak convergence of $\hat{\mathbb{Z}}_n$ follows from conditional weak convergence of $\mathbb{Z}_n'$.
For the latter, we first need to show unconditional convergence of $\mathbb{Z}_n'$ in our norm.
The following is a modification of \cite[Theorem 2.9.2]{vw1996}.

\begin{lem} \label{lem:multiplier:uncond}
Let $\xi_1,\dots,\xi_n$ be i.i.d.\ random variables with mean zero, variance $1$, and $\|\xi\|_{2,1}<\infty$, independent of $X_1,\dots,X_n$.
For a probability distribution $F$ on $\mathbb{R}$ such that $m(X)$ has a $(2+c)$th moment for $X\sim F$ and some $c>0$, the process $\mathbb{Z}_n'(x)\vcentcolon=n^{-1/2}\sum_{i=1}^n\xi_i[\mathbbm{1}\{X_i\leq x\}-F(x)]$ converges weakly to a tight limit process in $\mathbb{L}_\mu$ if and only if $\mathbb{Z}_n\vcentcolon=n^{-1/2}\sum_{i=1}^n[\mathbbm{1}\{X_i\leq x\}-F(x)]$ does.
In that case, they share the same limit processes.
\end{lem}

\begin{proof}[Proof of \cref{lem:multiplier:uncond}]
Marginal convergence and asymptotic equicontinuity of $\mathbb{Z}_n'$ follow from \cref{prop:A:F} and \cite[Theorem 2.9.2]{vw1996}.
It remains to show the equivalence of asymptotic equiintegrability of $\mathbb{Z}_n'$ and $\mathbb{Z}_n$.

Note that the proofs of \cite[Lemmas 2.3.1, 2.3.6, and 2.9.1 and Propositions A.1.4 and A.1.5]{vw1996} do not depend on the specificity of the norm $\|\cdot\|_{\mathcal{F}}$, but they continue to hold with $\|\cdot\|_{\mathbb{L}_\mu}$.
Given this, \cite[Lemma 2.3.11]{vw1996} also holds with $\|\cdot\|_{\mathbb{L}_\mu}$ (and $\|\cdot\|_{\mathbb{L}_{\mu,\delta_n}}$).
Finally, rewriting the proof of \cite[Theorem 2.9.2]{vw1996} in terms of $\|\cdot\|_{\mathbb{L}_\mu}$ yields the proof of this lemma.
\end{proof}

\begin{proof}[Proof of \cref{lem:multiplier:cond}]
By \cref{lem:multiplier:uncond}, $\mathbb{Z}_n'$ is asymptotically measurable.
Define a semimetric on $\mathbb{R}$ by
\(
	\rho(s,t)\vcentcolon=|F(s)-F(t)|\vee\int_s^t\sqrt{F(1-F)}|d\mu|.
\)
For $\delta>0$, $t_1<\cdots<t_p$ be such that $\rho(-\infty,t_1)\leq\delta$, $\rho(t_j,t_{j+1})\leq\delta$, and $\rho(t_p,\infty)\leq\delta$.
Define $\mathbb{Z}_\delta$ by
\[
	\mathbb{Z}_\delta(x)\vcentcolon=\begin{cases} 0 & x<t_1\text{ or }x\geq t_p, \\ \mathbb{Z}(t_i) & t_i\leq x\leq t_{i+1}, \, i=1,\dots,p-1. \end{cases}
\]
Define $\mathbb{Z}_{n,\delta}'$ analogously.
By the continuity and integrability of the limit process $\mathbb{Z}$, we have $\mathbb{Z}_\delta\to\mathbb{Z}$ in $\mathbb{L}_\mu$ almost surely as $\delta\to0$.
Therefore,
\(
	\sup_{h\in\text{BL}_1(\mathbb{L}_\mu)}\bigl|\mathbb{E}h(\mathbb{Z}_\delta)-\mathbb{E}h(\mathbb{Z})\bigr|\conv0$ as $\delta\to0.
\)
Second, by \cite[Lemma 2.9.5]{vw1996},
\(
	\sup_{h\in\text{BL}_1(\mathbb{L}_\mu)}\bigl|\mathbb{E}_\xi h(\mathbb{Z}_{n,\delta}')-\mathbb{E}h(\mathbb{Z}_\delta)\bigr|\conv0$ as $n\to\infty
\)
for almost every sequence $X_1,X_2,\dots$ and fixed $\delta>0$.
Since $\mathbb{Z}_\delta$ and $\mathbb{Z}_{n,\delta}'$ take only on a finite number of values and their tail values are zero, one can replace the supremum over $\text{BL}_1(\mathbb{L}_\mu)$ with a supremum over $\text{BL}_1(\mathbb{R}^p)$.
Observe that $\text{BL}_1(\mathbb{R}^p)$ is separable with respect to the topology of uniform convergence on compact sets; this supremum is effectively over a countable set, hence measurable.
Third,
\(
	\sup_{h\in\text{BL}_1(\mathbb{L}_\mu)}\bigl|\mathbb{E}_\xi h(\mathbb{Z}_{n,\delta}')-\mathbb{E}_\xi h(\mathbb{Z}_n')\bigr| \leq \sup_{h\in\text{BL}_1(\mathbb{L}_\mu)}\mathbb{E}_\xi\bigl|h(\mathbb{Z}_{n,\delta}')-h(\mathbb{Z}_n')\bigr|
	\leq \mathbb{E}_\xi\|\mathbb{Z}_{n,\delta}'-\mathbb{Z}_n'\|_{\mathbb{L}_\mu}^\ast \leq \mathbb{E}_\xi\|\mathbb{Z}_n'\|_{\mathbb{L}_{\mu,\delta}^\ast}.
\)
This implies that its outer expectation is bounded by $\mathbb{E}^\ast\|\mathbb{Z}_n'\|_{\mathbb{L}_{\mu,\delta}}$, which vanishes as $n\to\infty$ by the modified \cite[Lemma 2.9.1]{vw1996} as discussed in \cref{lem:multiplier:uncond}.
\end{proof}

\begin{proof}[Proof of \cref{lem:A:bootK}]
Noting $\mathbb{G}_n'(u)=0\vee[\mathbb{F}_n'(u)-\mathbb{F}_n'\circ\mathbb{F}_n'^{-1}(\alpha)]$, weak convergence of $\sqrt{n}(\mathbb{F}_n'-F)$ and $\sqrt{n}(\mathbb{G}_n'-G)$ follows from \cref{lem:multiplier:cond} (or \cite[Theorem 2.9.6]{vw1996}) and \cref{bZivHihLTt5o}.
\end{proof}

\begin{proof}[Proof of \cref{prop:A:bootstrap}]
With the remark below \cref{lem:A:Poissonization}, the proposition follows from \cref{lem:multiplier:cond,lem:A:bootK} and \cite[Theorem 3.9.11]{vw1996}.
\end{proof}


\appendix

\section*{Appendix}\label{app}


\subsection{Supporting Lemmas} \label{suppA}

\begin{lem} \label{lem:t97410h}
Let $F$ be a probability distribution on $\mathbb{R}$ and write $\tilde{F}(x) \vcentcolon= F(x) - \mathbbm{1}\{x \geq 0\}$.
For $p > 0$ we have (\ref{lemA1:1}) $\Leftrightarrow$ (\ref{lemA1:2}) $\Leftrightarrow$ (\ref{lemA1:3}) $\Rightarrow$ (\ref{lemA1:4}) $\Leftrightarrow$ (\ref{lemA1:5}), where
\begin{enumerate}[i.]
	\item \label{lemA1:1} $F$ has a $p$th moment;
	\item \label{lemA1:2} $Q$ is in $L_p(0,1)$;
	\item \label{lemA1:3} $|x|^{p-1} \tilde{F}$ is integrable;
	\item \label{lemA1:4} $|x|^p \tilde{F}$ converges to $0$ as $x \to \pm \infty$;
	\item \label{lemA1:5} $u^{1/p} (1-u)^{1/p} Q$ converges to $0$ as $u \to \{0,1\}$.
\end{enumerate}
\end{lem}

\begin{proof}
We proceed as follows: (\ref{lemA1:1}) $\Rightarrow$ (\ref{lemA1:4}), (\ref{lemA1:3}) $\Rightarrow$ (\ref{lemA1:4}), (\ref{lemA1:1}) $\Leftrightarrow$ (\ref{lemA1:3}), (\ref{lemA1:1}) $\Leftrightarrow$ (\ref{lemA1:2}), and (\ref{lemA1:4}) $\Leftrightarrow$ (\ref{lemA1:5}).

(\ref{lemA1:1}) $\Rightarrow$ (\ref{lemA1:4}).
For $M > 0$,
\(
	\int_{\mathbb{R}} |x|^p dF \geq \int_{[-M,M]} |x|^p dF + M^p |\tilde{F}(-M)| + M^p |\tilde{F}(M)|.
\)
Since the left\hyp{}hand side (LHS) is finite, one may take $M$ large enough that
\(
	\int_{\mathbb{R}} |x|^p dF - \int_{[-M,M]} |x|^p dF
\)
is arbitrarily small, which then bounds the two nonnegative terms. Hence $|x|^p \tilde{F}(x) \to 0$ as $x \to \pm\infty$.

(\ref{lemA1:3}) $\Rightarrow$ (\ref{lemA1:4}).
Suppose that $|x|^{p-1} |\tilde{F}|$ is integrable but $|x|^p F$ does not vanish as $x \to -\infty$, that is, there exist a constant $c > 0$ and a sequence $0 > x_1 > x_2 > \cdots \to -\infty$ such that $|x_i|^p F(x_i) \geq c$.
Since $F \to 0$, one may take a subsequence such that
\(
	|x_i|^p F(x_{i+1}) \leq 2^{-i}.
\)
By monotonicity of $F$,
\(
	p \int_{-\infty}^0 |x|^{p-1} F(x) dx \geq F(x_1) \int_{x_1}^0 p |x|^{p-1} dx + F(x_2) \int_{x_2}^{x_1} p |x|^{p-1} dx + \cdots
	= |x_1|^p F(x_1) + ( |x_2|^p - |x_1|^p ) F(x_2) + 
	\cdots
	\geq c + \sum_{i=1}^\infty (c - 2^{-i}) = \infty,
\)
a contradiction. Hence $|x|^p F$ vanishes.
Similarly $|x|^p (1-F) \to 0$, as $x \to \infty$.

(\ref{lemA1:1}) $\Leftrightarrow$ (\ref{lemA1:3}).
Note that $dF = d\tilde{F}$ for $x \neq 0$.
Integration by parts yields
\(
	\int_{\mathbb{R}} |x|^p dF 
	= \bigl[ |x|^p \tilde{F} \bigr]_{-\infty}^\infty + p \int_{-\infty}^\infty |x|^{p-1} |\tilde{F}| dx.
\)
If the LHS is finite (\ref{lemA1:1}), then the first term in the RHS is $0$ (\ref{lemA1:4}), hence the second term is finite (\ref{lemA1:3}).
Conversely, if the second term is finite (\ref{lemA1:3}), then the first term is $0$ (\ref{lemA1:4}), hence the LHS is finite (\ref{lemA1:1}).

(\ref{lemA1:1}) $\Leftrightarrow$ (\ref{lemA1:2}).
Since
\(
	\int_{\mathbb{R}} |x|^p dF = \int_0^1 |Q|^p du,
\)
the LHS is finite if and only if the right\hyp{}hand side (RHS) is.

(\ref{lemA1:4}) $\Leftrightarrow$ (\ref{lemA1:5}).
Let $u = F(x)$. Then, $\lim_{x\to-\infty} |x|^p \tilde{F} = \lim_{u\to0} (u^{1/p} Q)^p = 0$. Convergence of the other tail can be shown analogously.
\end{proof}

\begin{lem} \label{lem:joint}
Let $d_1$ and $d_2$ be metrics on $\mathbb{D}$.
Then, $X_\alpha$ converges weakly in $d_1$ and in $d_2$ to a limit $X$ that is tight in $d_1$ and in $d_2$ if and only if $X_\alpha$ converges weakly in $d_1\vee d_2$ to a limit $X$ that is tight in $d_1\vee d_2$.
\end{lem}

\begin{proof}
When we consider $\mathbb{D}$ in metrics $d_1$, $d_2$, and $d_1\vee d_2$, we denote them respectively by $\mathbb{D}_1$, $\mathbb{D}_2$, and $\mathbb{D}_{1\vee2}$.
Sufficiency is trivial.
Necessity is nontrivial since $C_b(\mathbb{D}_{1\vee2})$ is bigger than $C_b(\mathbb{D}_1)$ and $C_b(\mathbb{D}_2)$ in \cite[Definition 1.3.3]{vw1996}.
Note that the algebra generated by $C_b(\mathbb{D}_1)\cap C_b(\mathbb{D}_2)$ separates points of $\mathbb{D}_{1\vee2}$.
Therefore, in light of \cite[Lemma 1.3.13]{vw1996}, it suffices to show that tightness in $\mathbb{D}_1$ and in $\mathbb{D}_2$ implies tightness in $\mathbb{D}_{1\vee2}$.

Fix $\varepsilon>0$. Let $K_1$ and $K_2$ be sets compact under $d_1$ and $d_2$ respectively such that $\Pr(X\in K_1)\geq1-\varepsilon$ and $\Pr(X\in K_2)\geq1-\varepsilon$.
Then, $\Pr(X\in K_1\cap K_2)\geq1-2\varepsilon$.
Now we show $K_1\cap K_2$ is totally bounded under $d_1\vee d_2$.
Take $(t_1,\dots,t_p)$ and $(s_1,\dots,s_q)$ to be finitely many points such that $\varepsilon$\hyp{}$d_1$\hyp{}balls of $(t_1,\dots,t_p)$ cover $K_1$ and $\varepsilon$\hyp{}$d_2$\hyp{}balls of $(s_1,\dots,s_q)$ cover $K_2$.
Then choose a total of at most $pq$ points from each intersection of a $t$\hyp{}ball and an $s$\hyp{}ball, $(u_1,\dots,u_{pq})$, and consider $2\varepsilon$\hyp{}$(d_1\vee d_2)$\hyp{}balls around them.
Since every point in $K_1\cap K_2$ belongs to at least one intersection of a $t$\hyp{}ball and an $s$\hyp{}ball, these balls cover $K_1\cap K_2$ by the triangle inequality.
Therefore, $K_1\cap K_2$ is totally bounded, so its closure in $d_1\vee d_2$ is compact in $d_1\vee d_2$.
Since $\Pr(X\in\overline{K_1\cap K_2})\geq1-2\varepsilon$, $X$ is tight in $d_1\vee d_2$.
\end{proof}


\subsection{Application to Outlier Robustness Analysis} \label{suppB}

We construct a statistical test of outlier robustness analysis.
Recall our setup from \cref{exa:1} and consider the null hypothesis
\(
	H_0 : \|\beta_1-\beta_2\| \leq h
\)
for fixed $h\geq0$.
We assume that $h$ is a scalar while $\beta$ can be a vector, in which case $\|\cdot\|$ is the Mahalanobis distance between $\beta_1$ and $\beta_2$, that is, $[(\hat{\beta}_1-\hat{\beta}_2)'\Sigma^{-1}(\hat{\beta}_1-\hat{\beta}_2)]^{1/2}$ where $\Sigma$ is either an identity, the covariance matrix of $\hat{\beta}_1-\hat{\beta}_2$, or some other positive definite symmetric matrix.
The natural test statistic to use is $\|\hat{\beta}_1-\hat{\beta}_2\|$ ($\Sigma$ may be estimated consistently).
Let $\alpha\in(0,1)$ be the size of the test.
Our results imply that the variance $\Sigma$ of the difference $\hat{\beta}_1-\hat{\beta}_2$ can be estimated either by the analytic formula or by the bootstrap.
Note that if $h>0$, the null hypothesis is composite. 
The critical value $c_\alpha$ satisfies
\(
	\sup_{\|v\|\leq1} \Pr \bigl( \|hv+\xi\|^2 > c_\alpha \bigr) \leq \alpha
\)
for $\xi \sim N(0,\Sigma)$.
If $\beta$ is a scalar, it reduces to
$\Pr \bigl( (h+\xi)^2>c_\alpha \bigr) = \alpha$
for $\xi\sim N(0,\Var(\hat{\beta}_1-\hat{\beta}_2))$.

We reinvestigate the outlier robustness analysis in \cite{anrr2016}.
They tackle the long-standing question of whether and how democracy affects economic growth.
They find that after 25 years from permanent democratization, GDP per capita is about 20\% higher than without democratization, and check robustness of their results to outliers of the error term.
We revisit their fixed effects regressions and conduct the outlier robustness tests proposed above.

The first\hyp{}stage equation is
\[
	\text{\em Democracy}_{i,t} = \sum_{s=1}^4 \pi_s \text{\em Wave}_{i,t-s} + \sum_{s=1}^4 \phi_s \log \text{\em GDP}_{i,t-s} + \theta_i + \eta_t + v_{i,t},
\]
where $\text{\em Wave}_{i,t}$ is the instrumental variable (IV) constructed from the democracy indicators of nearby countries that share similar political history to country $i$.
The panel data is unbalanced; each country has a varying number of observations.
Let $t_i$ be the year of country $i$'s first appearance in the sample and $T_i$ be the number of observations country $i$ has.
Then, $i$'s time array spans $t_i,t_i+1,\dots,t_i+T_i-1$.

In addition to regression coefficients, \cite{anrr2016} report three parameters.
The {\em long\hyp{}run effect of democracy}, $\beta_5 \vcentcolon= \beta_0/(1-\beta_1-\beta_2-\beta_3-\beta_4)$, represents the impact on $\log \text{\em GDP}_{i,\infty}$ of the transition from non\hyp{}democracy $D_{i,t-1}=0$ to permanent democracy $D_{i,t+s}=1$ for every $s\geq0$.
The {\em effect of transition to democracy after 25 years}, $\beta_6 \vcentcolon= e_{25}$ where $e_j = \beta_0+\beta_1 e_{j-1}+\beta_2 e_{j-2}+\beta_3 e_{j-3}+\beta_4 e_{j-4}$ and $e_0=e_{-1}=e_{-2}=e_{-3}=0$, represents the impact on $\log \text{\em GDP}_{i,25}$ of the transition from $D_{i,t-1}=0$ to $D_{i,t+s}=1$ for $0\leq s\leq25$.
{\em Persistence of the GDP process}, $\beta_7 \vcentcolon= \beta_1+\beta_2+\beta_3+\beta_4$, represents how persistently a unit change in $\log \text{\em GDP}$ remains.

To check robustness to outliers, \cite{anrr2016} carry out same regression excluding observations that have large residuals.
For notational convenience, let
\begin{align*}
	x_{i,t}&\vcentcolon=(\text{\em Democracy}_{i,t},\log\text{\em GDP}_{i,t-1},\cdots,\log\text{\em GDP}_{i,t-4}, \\
	&\hspace{140pt}\mathbbm{1}_{i=1},\cdots,\mathbbm{1}_{i=N},\mathbbm{1}_{t=0},\cdots,\mathbbm{1}_{t=T})', \\
	\beta&\vcentcolon=(\beta_0,\beta_1,\cdots,\beta_4,\alpha_1,\cdots,\alpha_N,\delta_1,\cdots,\delta_T)', \\
	z_{i,t}&\vcentcolon=(\text{\em Wave}_{i,t-1},\cdots,\text{\em Wave}_{i,t-4},\log\text{\em GDP}_{i,t-1},\cdots,\log\text{\em GDP}_{i,t-4}, \\
	&\hspace{140pt}\mathbbm{1}_{i=1},\cdots,\mathbbm{1}_{i=N},\mathbbm{1}_{t=0},\cdots,\mathbbm{1}_{t=T})', \\
	\pi&\vcentcolon=(\pi_1,\cdots,\pi_4,\phi_1,\cdots,\phi_4,\theta_1,\cdots,\theta_N,\eta_1,\cdots,\eta_T)'.
\end{align*}
Outliers are defined by $|\hat{\varepsilon}_{i,t}|\geq1.96\,\hat{\sigma}_\varepsilon$, where $\hat{\sigma}_\varepsilon$ is the estimated homoskedastic standard error of $\varepsilon$,%
\footnote{The purpose of $\hat{\sigma}_\varepsilon$ is normalization.
\cite{anrr2016} do use heteroskedasticity-robust standard errors for inference.}
\(
	\hat{\sigma}_\varepsilon^2 \vcentcolon= \frac{1}{n}\sum_{i=1}^n\frac{1}{T_i}\sum_{t=t_i}^{t_i+T_i-1}(y_{i,t}-x_{i,t}'\hat{\beta})^2,
\)
and, for the IV model, also by $|\hat{v}_{i,t}|\geq1.96\,\hat{\sigma}_v$, where
\(
	\hat{\sigma}_v^2 \vcentcolon= \frac{1}{n}\sum_{i=1}^n\frac{1}{T_i}\sum_{t=t_i}^{t_i+T_i-1}(x_{i1,t}-z_{i,t}'\hat{\pi})^2.
\)
This means that they are concerned with whether tail observations of the GDP might have disproportionate effects on estimates.
Defining outliers based on $\hat{\varepsilon}$, not $y$, even if they are interested in the effects of outliers of the GDP, is reasonable since sample selection based on $\hat{\varepsilon}$ does not affect the true parameters under some conditions while selection on $y$ certainly does.

Let $\mathbb{F}_n$ be the vector of empirical distribution functions of $\frac{1}{T_i}\sum_t x_{i,t}y_{i,t}$ and $\mathbb{Q}_n$ the vector of marginal empirical quantile functions of $\frac{1}{T_i}\sum_t x_{i,t}y_{i,t}$.
Note that, with $w_{i,t} = \mathbbm{1}\{|\hat{\varepsilon}_{i,t}|\geq 1.96 \hat{\sigma}_\varepsilon\}$, the full\hyp{}sample and outlier\hyp{}removed OLS estimators are
\begin{align*}
	\hat{\beta}^1_{\text{OLS}} &= \Biggl( \frac{1}{n} \sum_{i=1}^n \frac{1}{T_i} \sum_{t=t_i}^{t_i+T_i-1} x_{i,t} x_{i,t}' \Biggr)^{-1} \frac{1}{n} \sum_{i=1}^n \frac{1}{T_i} \sum_{t=t_i}^{t_i+T_i-1} x_{i,t} y_{i,t} \\
	&= \Biggl( \frac{1}{n} \sum_{i=1}^n \frac{1}{T_i} \sum_{t=t_i}^{t_i+T_i-1} x_{i,t} x_{i,t}' \Biggr)^{-1} \int_0^1 \mathbb{Q}_n(u) du, \\
	\hat{\beta}^2_{\text{OLS}} &= \Biggl( \frac{1}{n} \sum_{i=1}^n \frac{1}{T_i} \sum_{t=t_i}^{t_i+T_i-1} x_{i,t} x_{i,t}' w_{i,t} \Biggr)^{-1} \frac{1}{n} \sum_{i=1}^n \frac{1}{T_i} \sum_{t=t_i}^{t_i+T_i-1} x_{i,t} y_{i,t} w_{i,t} \\
	&= \Biggl( \frac{1}{n} \sum_{i=1}^n \frac{1}{T_i} \sum_{t=t_i}^{t_i+T_i-1} x_{i,t} x_{i,t}' w_{i,t} \Biggr)^{-1} \int_0^1 \mathbb{Q}_n(u) d\mathbb{K}_n(u),
\end{align*}
where $\mathbb{K}_n$ is the vector of measures whose $j$th element assigns density
\[
	\frac{\sum_t x_{i,t,j}y_{i,t}w_{i,t}}{\sum_t x_{i,t,j}y_{i,t}} \ \text{to} \ u\in\biggl(\mathbb{F}_{n,j}\biggl(\frac{1}{T_i}\sum_t x_{i,t,j}y_{i,t}\biggr)-\frac{1}{n},\mathbb{F}_{n,j}\biggl(\frac{1}{T_i}\sum_t x_{i,t,j}y_{i,t}\biggr)\biggr].
\]
Assume that $\frac{1}{T_i}\sum_t x_{i,t} y_{i,t}$ has smooth cdfs with $(2+c)$th moments for some $c>0$ and $\hat{\sigma}_\varepsilon$ has a well-defined limit.
Then, our results imply that the joint distribution of $\hat{\beta}^1_{\text{OLS}}$ and $\hat{\beta}^2_{\text{OLS}}$ converges and can be estimated by nonparametric bootstrap.
Similar arguments apply also to the IV estimators.

In a simple case where $\varepsilon$ and $v$ are independent of covariates, outlier removal will not change the true coefficients. So, it seems sensible to set $h=0$, the most conservative choice.
Thus, we test the hypothesis $H_0 : \beta_j^1 = \beta_j^2$.

We carry out nonparametric bootstrap across $i$.
All fixed effects are replaced by dummy variables.
Each draw of country $i$ adds $T_i$ observations to the bootstrap sample; equivalently, we treat each sum, $\frac{1}{T_i}\sum_t x_{i,t}y_{i,t}$, $\frac{1}{T_i}\sum_t y_{i,t}y_{i,t-s}$, and $\frac{1}{T_i}\sum_t z_{i,t}y_{i,t}$, as one observation in order to exploit the i.i.d.\ structure.
Bootstrap runs for 10,000 iterations, in each of which we draw 175 countries for OLS and 174 for IV with replacement.



\begin{sidewaystable}
\singlespacing
\centering
\caption{Comparison of formal and heuristic $p$\hyp{}values for the outlier robustness test in \cite{anrr2016}.}
\label{tbl:joint}
\small
\begin{tabular}{lcccccccccccccc}
\hline\hline
& && \multicolumn{5}{c}{Estimate} && \multicolumn{5}{c}{$p$\hyp{}value for $H_0 : \beta_j^1 = \beta_j^2$} \\
\cline{4-8} \cline{10-14}
& && \multicolumn{2}{c}{OLS} && \multicolumn{2}{c}{IV} && \multicolumn{2}{c}{OLS} & & \multicolumn{2}{c}{IV} \\
\cline{4-5} \cline{7-8} \cline{10-11} \cline{13-14}
& Notation && (1) & (2) && (3) & (4) && (5) & (6) && (7) & (8) \\
\hline
Democracy & $\beta_0$ &&
	$\phantom{0}$0.79 & $\phantom{0}$0.56 && $\phantom{0}$1.15 & $\phantom{0}$0.66 &&
	0.15 & 0.32 && 0.20 & 0.41 \\
	& &&
	$\phantom{0}$(0.23) & $\phantom{0}$(0.20) && $\phantom{0}$(0.59) & $\phantom{0}$(0.44) && \\
log GDP first lag & $\beta_1$ &&
	$\phantom{0}$1.24 & $\phantom{0}$1.23 && $\phantom{0}$1.24 & $\phantom{0}$1.23 &&
	0.60 & 0.74 && 0.70 & 0.82 \\
	& &&
	$\phantom{0}$(0.04) & $\phantom{0}$(0.02) && $\phantom{0}$(0.04) & $\phantom{0}$(0.03) && \\
log GDP second lag & $\beta_2$ &&
	$\phantom{0}\mathllap{-}$0.21 & $\phantom{0}\mathllap{-}$0.20 &&
	$\phantom{0}\mathllap{-}$0.21 & $\phantom{0}\mathllap{-}$0.20 &&
	0.75 & 0.85 && 0.84 & 0.91 \\
	& &&
	$\phantom{0}$(0.05) & $\phantom{0}$(0.03) && $\phantom{0}$(0.05) & $\phantom{0}$(0.04) && \\
log GDP third lag & $\beta_3$ &&
	$\phantom{0}\mathllap{-}$0.03 & $\phantom{0}\mathllap{-}$0.03 &&
	$\phantom{0}\mathllap{-}$0.03 & $\phantom{0}\mathllap{-}$0.03 &&
	0.97 & 0.98 && 0.89 & 0.93 \\
	& &&
	$\phantom{0}$(0.03) & $\phantom{0}$(0.02) && $\phantom{0}$(0.03) & $\phantom{0}$(0.03) && \\
log GDP fourth lag & $\beta_4$ &&
	$\phantom{0}\mathllap{-}$0.04 & $\phantom{0}\mathllap{-}$0.03 &&
	$\phantom{0}\mathllap{-}$0.04 & $\phantom{0}\mathllap{-}$0.03 &&
	0.26 & 0.45 && 0.28 & 0.46 \\
	& &&
	$\phantom{0}$(0.02) & $\phantom{0}$(0.01) && $\phantom{0}$(0.02) & $\phantom{0}$(0.02) && \\
[0.5em]
Long-run effect of democracy & $\beta_5$ &&
	21.24 & 19.32 && 31.52 & 22.63 &&
	0.72 & 0.79 && 0.46 & 0.63 \\
	& &&
	$\phantom{0}$(7.32) & $\phantom{0}$(8.54) && (18.49) & (18.14) && \\
Effect of democracy after 25 years & $\beta_6$ &&
	16.90 & 13.00 && 24.87 & 15.47 &&
	0.29 & 0.46 && 0.27 & 0.49 \\
	& &&
	$\phantom{0}$(5.32) & $\phantom{0}$(5.02) && (13.53) & (10.82) && \\
Persistence of GDP process & $\beta_7$ &&
	$\phantom{0}$0.96 & $\phantom{0}$0.97 && $\phantom{0}$0.96 & $\phantom{0}0.97$ &&
	0.00$\mathrlap{02}$ & 0.12 && 0.00$\mathrlap{4}$ & 0.20 \\
	& &&
	$\phantom{0}$(0.01) & $\phantom{0}\mathllap{(}$0.00$\mathrlap{5)}$ && $\phantom{0}$(0.01) & $\phantom{0}\mathllap{(}$0.00$\mathrlap{5)}$ && \\
[0.5em]
\hline
Number of observations & && 6,336 & 6,044 && 6,309 & 5,579 \\
Number of countries & $n$ && 175 & 175 && 174 & 174 \\
Average number of years & $\overline{T_i}$ && 36.2 & 34.5 && 36.3 & 32.1 \\
\hline\hline
\end{tabular}
\\ {\footnotesize * (1,3) Baseline estimates; (2) Estimates with $|\hat{\varepsilon}_{i,t}|<1.96\,\hat{\sigma}_\varepsilon$; (4) Estimates with $|\hat{\varepsilon}_{i,t}|<1.96\,\hat{\sigma}_\varepsilon$ and $|\hat{v}_{i,t}|<1.96\,\hat{\sigma}_v$; (5,7) $p$\hyp{}values of the formal tests that use the standard errors of $\hat{\beta}_j^1-\hat{\beta}_j^2$; (6,8) ``$p$\hyp{}values'' of the heuristic tests that use the marginal standard errors of $\hat{\beta}_j^1$. Some numbers in Columns (1,2) differ from \cite{anrr2016} since we use our own bootstrap to compute standard errors.}
\end{sidewaystable}


\begin{figure}
\centering
\begin{subfigure}[t]{0.45\textwidth}
\centering
\includegraphics[width=0.93733\textwidth]{
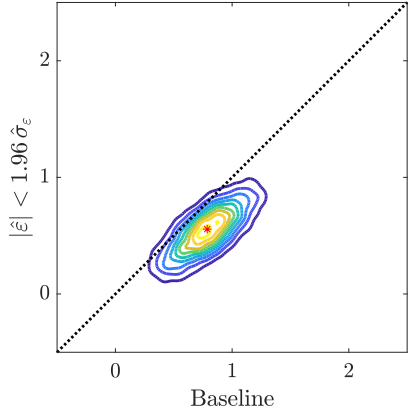}
\caption{Distribution of full\hyp{}sample and outlier\hyp{}removed OLS estimators for the effect of democracy $\beta_0$. $p=0.15$.}
\label{fig-ols-i-dem}
\end{subfigure}
\qquad
\begin{subfigure}[t]{0.45\textwidth}
\centering
\includegraphics[width=\textwidth]{
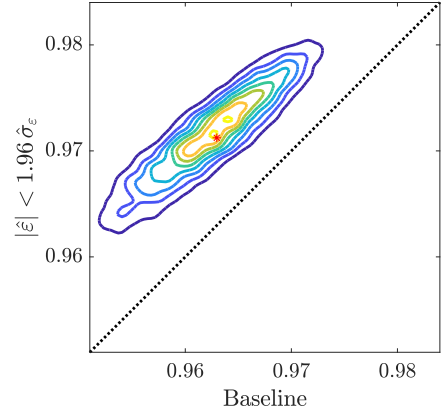}
\caption{Distribution of full\hyp{}sample and outlier\hyp{}removed OLS estimators for persistence of GDP $\beta_7$. $p=0.0002$.}
\label{fig-ols-i-per}
\end{subfigure}

\begin{subfigure}[t]{0.45\textwidth}
\centering
\includegraphics[width=0.93733\textwidth]{
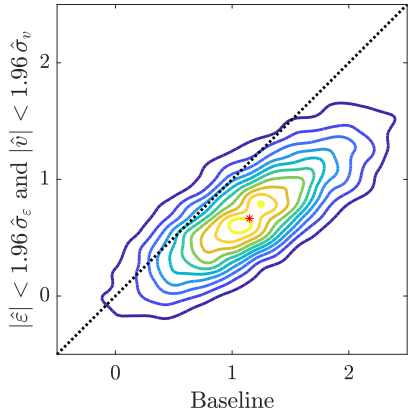}
\caption{Distribution of full\hyp{}sample and outlier\hyp{}removed IV estimators for the effect of democracy $\beta_0$. $p=0.20$.}
\label{fig-iv-i-dem}
\end{subfigure}
\qquad
\begin{subfigure}[t]{0.45\textwidth}
\centering
\includegraphics[width=\textwidth]{
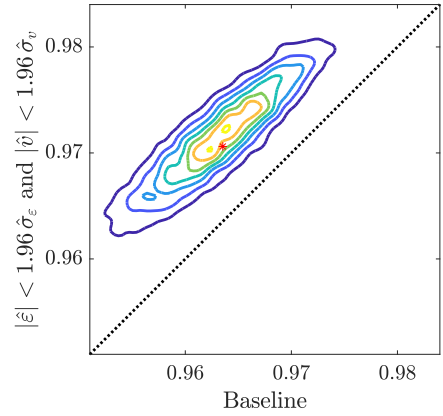}
\caption{Distribution of full\hyp{}sample and outlier\hyp{}removed IV estimators for persistence of GDP $\beta_7$. $p=0.004$.}
\label{fig-iv-i-per}
\end{subfigure}

\caption{Joint distributions of full\hyp{}sample and outlier\hyp{}removed OLS and IV estimators for \cite{anrr2016}. Outliers are defined by $|\hat{\varepsilon}_{i,t}|\geq1.96\,\hat{\sigma}_\varepsilon$ or $|\hat{v}_{i,t}|\geq1.96\,\hat{\sigma}_v$. The black dotted lines indicate the 45 degree. Nonparametric bootstrap runs for 10,000 iterations, randomly sampling $i$. The contours are of kernel density estimators.}
\end{figure}

Most of our reexamination reconfirms robustness to outliers even though we set $h=0$.
The exception is persistence of the GDP process, of which we reject the hypothesis of no change.
\Cref{tbl:joint} lists the estimates and the $p$\hyp{}values for our tests.
Column 1 is the baseline OLS estimates and column 2 is OLS excluding $|\hat{\varepsilon}_{i,t}|\geq1.96\hat{\sigma}_\varepsilon$.
Column 3 is the baseline IV estimates and column 4 is IV excluding observations satisfying either $|\hat{\varepsilon}_{i,t}|\geq1.96\hat{\sigma}_\varepsilon$ or $|\hat{v}_{i,t}|\geq1.96\hat{\sigma}_v$.
Columns 5 to 8 illustrate the utility of our formal tests of outlier robustness.
Column 5 gives the $p$\hyp{}values of the hypotheses that the two OLS coefficients are identical, using the standard error of the difference of two estimators estimated by bootstrap.
Column 6 gives the ``$p$\hyp{}values" of the same hypotheses, but heuristically using the standard error of the marginal distribution of the baseline OLS estimates.
Columns 7 and 8 are the corresponding $p$\hyp{}values for IV coefficients.
The identity of persistence of the GDP process is rejected in formal tests while accepted in heuristic tests at 5\%.
We note that the magnitudes of persistence are very close (0.96 and 0.97), so if we allow bias $h$ of, say, $0.01$, the hypothesis will not be rejected.

Positive correlation of baseline and outlier\hyp{}removed estimators is visualized by the bootstrap distributions.
\Cref{fig-ols-i-dem,fig-ols-i-per} illustrate the joint distributions of baseline and outlier\hyp{}removed OLS estimators, $(\hat{\beta}_0^1, \hat{\beta}_0^2)$ and $(\hat{\beta}_7^1, \hat{\beta}_7^2)$.
\Cref{fig-iv-i-dem,fig-iv-i-per} are the corresponding figures for IV.
The contour plots are based on the kernel density estimators of the bootstrap distributions.
The estimators are positively correlated as anticipated by the fact that they share much of the samples.
Graphically, the tests examine if the red stars (estimators) are close to the 45 degree lines (black dotted lines).


\section*{Acknowledgements}
I thank Anna Mikusheva, Elena Manresa, Kengo Kato, Rachael Meager, Matthew Masten, Abhijit Banerjee, Daron Acemoglu, Isaiah Andrews, Hideatsu Tsukahara, Hidehiko Ichimura, Victor Chernozhukov, Jerry Hausman, Whitney Newey, Alberto Abadie, Joshua Angrist, and Brendan K.\ Beare for helpful comments. Daron Acemoglu and Pascual Restrepo kindly shared data and codes of their paper. This work is supported by the Richard N.\ Rosett Faculty Fellowship and the Liew Family Faculty Fellowship at the University of Chicago Booth School of Business.






\bibliographystyle{imsart-number}
\bibliography{ims-bib}

\end{document}